\documentclass[12pt,letterpaper,reqno]{amsart}

\usepackage{mathrsfs,enumerate}
\usepackage{amsmath,amsfonts,amsthm,amssymb,amscd,mathtools}
\usepackage{tikz-cd}
\usetikzlibrary{decorations.pathreplacing}
\usepackage{latexsym}

\usepackage{libertinus}
\usepackage[T1]{fontenc}

\definecolor{cadmiumgreen}{rgb}{0.0, 0.42, 0.24}
\usepackage[
colorlinks, citecolor=cadmiumgreen,
pagebackref,
pdfauthor={Jeffrey C. Lagarias, David Harry Richman}, 
pdfstartview ={FitV},
]{hyperref}
\hypersetup{
pdftitle={The family of a-floor quotient partial orders},
}

\usepackage[
	backrefs,
	msc-links,
	nobysame,
	lite,
    initials,
]{amsrefs} 

\usepackage{chngcntr}

\newtheorem{thm}{Theorem}[section]

\newtheorem{lem}[thm]{Lemma}
\newtheorem{prop}[thm]{Proposition}

\theoremstyle{definition}
\newtheorem{defi}[thm]{Definition}
\newtheorem{exa}[thm]{Example}

\theoremstyle{remark}
\newtheorem{rmk}[thm]{Remark}

\newcommand{\RR}{\ensuremath{\mathbb{R}}}

\newcommand{\ZZ}{\ensuremath{\mathbb{Z}}}
\newcommand{\QQ}{\mathbb{Q}}
\newcommand{\NN}{{\mathbb{N}}}
\newcommand{\NNplus}{{\mathbb{N}^+}}

\newcommand{\sA}{{\mathcal{A}}}

\newcommand{\sD}{{\mathcal D}}

\newcommand{\sI}{{\mathcal I}}
\newcommand{\sK}{{\mathcal K}}
\newcommand{\sM}{{\mathcal M}}

\newcommand{\sP}{{\mathcal P}}

\newcommand{\sQ}{{\mathcal Q}}

\newcommand{\sS}{{\mathcal S}}

\newcommand{\dd}{d}
\newcommand{\asafe}{\texorpdfstring{$a$}{a}}

\newcommand{\floor}[1]{\left\lfloor{#1}\right\rfloor}

\newcommand{\floorfrac}[2]{\floor{\frac{#1}{#2}}}
\newcommand{\floorslash}[2]{\floor{{#1} / {#2}}}

\newcommand{\angles}[1]{\left\langle{#1}\right\rangle}
\newcommand{\verts}[1]{\left\vert{#1}\right\vert}

\newcommand{\J}{{\Phi}} 
\newcommand{\AD}{{\,\preccurlyeq_{1}\,}}
\newcommand{\ADa}{{\,\preccurlyeq_{a}\,}}
\newcommand{\ADof}[1]{{\,\preccurlyeq_{#1}\,}}

\newcommand{\notAD}{{\,\not\preccurlyeq_{1}\,}}
\newcommand{\notADa}{{\,\not\preccurlyeq_{a}\,}}
\newcommand{\muAD}{{\mu_{1}}}
\newcommand{\muADa}{{\mu_{a}}}
\newcommand{\muADof}[1]{{\mu_{#1}}}
\newcommand{\mufquo}[1]{{\mu_{#1}}}

\newcommand{\ADset}{\sQ}
\newcommand{\ADaset}{\sQ_a}

\newcommand{\ADsetsmall}{\sQ_1^-}
\newcommand{\ADsetlarge}{\sQ_1^+}

\newcommand{\ADasetsmall}{\sQ_a^-}
\newcommand{\ADasetlarge}{\sQ_a^+}

\newcommand{\Dset}{\sD}
\newcommand{\ADmult}{\sM}

\newcommand{\multcount}{\sigma}

\newcommand{\divides}{\,\mid\,}
 
\newcommand{\fquo}[1]{\preccurlyeq_{#1}}

\newcommand{\Mobius}{M\"{o}bius}

\begin{document}
\title{The family of $a$-floor quotient partial orders} 

\author{Jeffrey C. Lagarias}
\address{Dept.\ of Mathematics, University of Michigan, Ann Arbor MI 48109-1043} 
\email{lagarias@umich.edu}

\author{David Harry Richman}
\address{Dept.\ of Mathematics, University of Washington,
Seattle, WA  98195}
\email{hrichman@uw.edu}

\date{March 7, 2024}

\begin{abstract}
An approximate divisor order is a partial order on the positive integers $\NNplus$ 
that refines the divisor order and is refined by the additive total order. 
A previous paper studied such a partial order on $\NNplus$, produced using the floor function. 
A positive integer $d$ is a floor quotient of $n$, denoted $d \AD n$, if there is a positive integer $k$ such that $d = \floorslash{n}{k}$.
The floor quotient relation defines a partial order on the positive integers.
This paper studies a  family of partial orders, the $a$-floor quotient relations $\ADa$, for $a \in \NNplus$, 
which interpolate between the floor quotient order and the divisor order on $\NNplus$.  
The paper studies the internal structure of these orders.
\end{abstract} 
\maketitle

\tableofcontents


%
%
\section{Introduction}\label{sec:1}

Denote the positive integers $\NNplus = \{ 1, 2, 3, \ldots\}$.
We call a partial order 
$\sP = (\NNplus, \preccurlyeq)$
on the positive integers  an
{\em approximate divisor order}  if:
\begin{enumerate}
\item
It refines the (multiplicative) divisor partial order $\sD =(\NNplus, \divides)$, i.e.
\[
\dd \divides n \qquad\Rightarrow\qquad \dd \preccurlyeq n,
\]
in which $\dd \divides n$  denotes $\dd$ divides $n$.
\item
It is refined by the additive total  order $\sA = ( \NNplus, \le)$, i.e.
\[
\dd \preccurlyeq n \qquad\Rightarrow\qquad \dd \leq n .
\]
\end{enumerate}
In a  previous paper \cite{LagR:23a} we studied a partial order on 
the positive integers $\NNplus$ called the floor quotient partial order, defined using
the floor function $\floor{x}$, which sends $x$ to the greatest integer no larger than $x$.

\begin{defi}
A positive integer $\dd$ is called a {\em floor quotient} of $n$, 
written $\dd \AD n$,  
if $\dd = \floor{ \frac{n}{k} }$ for some positive  integer $k$.
\end{defi}

The floor quotient order is an approximate divisor order. It has many more approximate divisors
of a given integer $n$ than the usual divisor order, about $2 \sqrt{n}$, 
as noted by Cardinal \cite[Prop. 4]{Cardinal:10} and Heyman \cite{Heyman:19}.

%
%
\subsection{The a-floor quotient order}\label{subsec:10} 

This paper studies an infinite family of approximate divisor orders, constructed from the floor quotient relation, 
called {\em $a$-floor quotients}, where $a\geq 1$ is an integer parameter.

\begin{defi}\label{def:a-floor} 
A positive integer $\dd$ is called an {\em $a$-floor quotient} of $n$, 
written $\dd \ADa n$,  if and only if $ad \AD an$ holds. 
\end{defi} 

The floor quotient relation is the case $a=1$ of this relation. 
The $1$-floor quotient relation is fundamental in that all the $a$-floor quotient relations 
are directly constructible from it; they are embedded inside the  $1$-floor quotient relation, via 
the dilation  map $d \to ad$ used in its definition.


\begin{thm}\label{thm:a-approx-order} 
The $a$-floor quotient relation $\ADa$
defines a partial order  $\ADaset := ( \NNplus, \ADa)$.
For each $a \ge 1$, the  partial order $\ADaset$ is an approximate divisor order.
\end{thm} 

Theorem \ref{thm:a-approx-order} follows as a consequence of  a more detailed result (Theorem~\ref{thm:a-equiv-properties})   
showing the equivalence of six  properties characterizing  $a$-floor quotients.  These six properties parallel
the six equivalence for the $1$-floor quotient proved in \cite[Theorem 3.1]{LagR:23a}.
The first two of these are: 
\begin{enumerate}
\item (Cutting property)  
 $a\dd = \floor{ \frac{an}{k}}$ for some integer $k$. 

\item (Covering property)  
For the half-open unit interval $I_\dd= [\dd, \dd + 1/a)$, there is some integer
$k$ such that the interval
$k I_\dd$ covers $I_n$,
i.e.
$ k  [\dd, \dd + 1/a) \supseteq [n, n + 1/a).$
\end{enumerate}

The partial order axioms for $\ADset_a = (\NNplus, \ADa)$ follow 
 easily from the covering property (2).
The cutting property (1) easily shows that the {floor quotient} order refines the
divisor partial order, and  is refined by the additive total order.

If $d$ is an $a$-floor quotient of $n$, then we will call $n$ an {\em $a$-floor  multiple} of $d$; 
we  will also call any associated value $k$ such that $an= \floorfrac{ad}{k}$ an {\em $a$-cutting length} of $(\dd,n)$.
There may be more than one $a$-cutting length  $k$ yielding a given $a$-floor quotient pair  $(\dd, n)$.
The  $a$-floor quotient relation drops the extra information contained in the value of $k$.


The  study of the $a$-floor quotients is motivated by the fact that 
the divisor order has a dilation symmetry preserving the order structure:  
$d \mid n$ implies $ad \mid an$ for all $a \ge 1$.
The floor quotient order $\AD$ does not have this dilation symmetry,
and applying the $a$-dilation results in a new partial order, the $a$-floor quotient order.
The next two subsections state some of the main results concerning this family of partial orders.

%
%
\subsection{Hierarchy of \asafe-floor quotient orders}\label{subsec:11}

A main result states that the family of $a$-floor quotient orders for $a\in \NNplus$ interpolates monotonically between the floor quotient order 
$\ADset_1 = (\NNplus, \AD)$ 
and the divisor partial order $\sD = (\NNplus, \divides)$.

\begin{thm}[$a$-floor quotient order hierarchy]
\label{thm:hierarchy-00}
\hfill
\begin{enumerate}[(1)]
\item If $a < b$ (in the additive order) then the $a$-floor quotient partial order $\ADset_a$ strictly refines the
$b$-floor quotient partial order  $\ADset_b$. 
That is, the  $a$-floor quotient orders form a nested decreasing sequence of order relations on $\NNplus$ as the parameter $a$ increases.  

\item The intersection of all the $a$-floor quotient partial order relations, 
$\bigcap_{a = 1}^{\infty} \ADset_a$,
is  the divisor relation  $\sD = (\NNplus, \divides)$.
\end{enumerate}
\end{thm}

\noindent Theorem \ref{thm:hierarchy-00} follows directly from Theorem \ref{thm:hierarchy-11}, proved in Section \ref{subsec:33}.

Section \ref{sec:4}  gives results  on the rate of convergence of an initial interval $\ADset_a[1,n]$ of the $a$-floor quotient
order to the divisor order, as $a$ increases.  (Recall that an {\em interval} $\sP[d, n]$ of a partial order $\sP$ on $\NNplus$
is $\sP[d, n] = \{ e: d\preccurlyeq_{\sP} e \preccurlyeq_{\sP} n\}$, and is empty if $d \not\preccurlyeq_{\sP} n$.) 
The next result determines the range of $a$ for which a  given $(d, n)$ belongs to  the $a$-floor quotient partial order.
To state it, we define the {\em scaling set} $\sS(d, n)$ of a given $(d, n)$ as 
\[ 
	\sS(d, n) := \{ a\in \NNplus  : \,\, d \fquo{a} n  \} .
\]
We show the following result.

%
%

\begin{thm}[Structure of scaling sets]
\label{thm:scaling-trichotomy0}
For positive integers $d$ and $n$, the set $\sS(d,n)$ of all $a$ where $d \ADa n$ is either $\NNplus$, or
an initial  interval $\sA[1, m]$  of the additive order $\sA = (\NNplus, \leq)$, or the empty set. 
The  following trichotomy holds.
\hfill
\begin{enumerate}[(i)]
\item
If $d \mid n$, then $d \ADa n$ 
for all $a \in \NNplus$.

\item
If  $d \nmid n$ but $d \AD n$, then $d \ADa n$ if and only if $\displaystyle a < \frac{\floor{n/d}}{\delta(d, n)}$,
where $\delta(d,n) = n - d \floor{n/d}$.

\item 
If $d \notAD n$, then $d \notADa n$ for all $n \in \NNplus$.
\end{enumerate}
\end{thm}
\noindent This result is proved in Section 4 as Theorem \ref{thm:scaling-trichotomy}. 

We use Theorem \ref{thm:scaling-trichotomy0} to obtain a result on stabilization for an interval of the order.

 %
\begin{thm}[$a$-floor quotient  order stabilization]
\label{thm:interval-hierarchy-10} 
The $a$-floor quotient partial order on the  interval $\ADaset[d,n]$
is identical  to the  divisor partial order on the  interval $\sD[d,n]$
when $a \ge \floorfrac{n}{d}$.
That is,
\[
	(\ADaset [d,n], \ADa)  = ( \sD[d,n], \mid \,)  \qquad \text{for all}  \quad a \ge \floor{ \frac{n}{d}}.
\]
If in addition $d \mid n$,
then  the $a$-floor quotient order on the interval $\ADaset[d, n]$ is  identical  to 
the divisor partial order on $\sD[d,n]$
when  $a \ge \floorfrac{n}{d+1}$. 
\end{thm}  
\noindent The result is proved in Section 4 as Theorem \ref{thm:interval-hierarchy-1}.

We may reinterpret these results as making statements about the
two-variable zeta functions of these partial orders as $a$ varies.
The two-variable partial order zeta function $\zeta_{\ADset_a}: {\NNplus \times \NNplus} \to \ZZ$
of the $a$-floor quotient order $\ADset_a$
 is given by
\begin{equation}
	\zeta_{\ADset_a}(d, n) = \begin{cases}
	1 & \quad \mbox{if} \quad d \ADa n, \\
	0 & \quad  \mbox{otherwise}.
	\end{cases}
\end{equation} 
In the sequel we abbreviate notation to rename this function $\zeta_a(d, n)$. 
Theorem \ref{thm:hierarchy-00}  asserts  that  as $a \to \infty$,
the zeta functions  $\zeta_a(\cdot, \cdot)$ converge pointwise 
to the two-variable partial order zeta function of the divisor order $\sD$, which is
\begin{equation}
\zeta_{\sD}(d, n) = \begin{cases}
1 & \quad \mbox{if} \quad d\mid n\\
0 & \quad  \mbox{otherwise}.
\end{cases}
\end{equation} 
Since Theorem \ref{thm:hierarchy-00} says the  divisor order $\sD$ corresponds to the limit  $a= +\infty$,  we may rename  it  $\sQ_{\infty}$,
and  abbreviate  notation of $\zeta_{\sD}(d, n)$ to  $ \zeta_{\infty}(d, n)$.
Theorem \ref{thm:interval-hierarchy-10}  asserts that the convergence for $\zeta_a$ to $\zeta_{\infty}$ is  monotonically downward pointwise as $a$ increases,
with  the convergence on $\{1, \ldots, n\}$ stabilized for all $a \ge n$. 
 
A related topic concerns  the behavior of  the \Mobius{} function of the $a$-floor quotient partial orders.
The partial order \Mobius{} function is the inverse of the zeta function in the incidence algebra of the partial order, under convolution. 
The results above imply that the \Mobius{} functions of the $a$-floor quotient partial orders,
denoted $\muADa(\cdot, \cdot)$, 
converge pointwise (as $a$ increases) to the \Mobius{} function $\mu_{\sD}(\cdot, \cdot)$ of the divisor order.
The function $\mu_{\sD}$ is given by
\begin{equation}
\mu_{\sD}(d, n) = \begin{cases}
\mu(\frac{n}{d}) & \quad \mbox{if} \quad d \mid n,\\
0 & \quad  \mbox{otherwise},
\end{cases}
\end{equation} 
where $\mu(n)$ is the classical \Mobius{} function.
Since the $a$-floor quotient relation $\ADa$ converges to
the divisor relation as $a \to \infty$, 
for each $(m,n)$ the \Mobius{} value $\muADa(m,n)$
converges to $\mu_\sD(m,n)$ as $a \to \infty$. 
The  convergence property as $a \to \infty$
motivates  study of the $a$-floor quotient \Mobius{} functions as a possible source of new insight
into the behavior of the classical \Mobius{} function.

We do not  study  the \Mobius{} function $\muADa$ in any detail in this paper, making only a few
remarks in Section \ref{sec:concluding}. The \Mobius{} value $\mu_{a}(d,n)$ always stabilizes  at 
$\mu_{\sD}(d,n)$ for $a \ge \floorfrac{n}{d}$,
Example~\ref{exa:mu-13} and Example  \ref{exa:mu-21} show that  the convergence 
to $\mu_{\sD}(d,n)$ is generally not monotonic in $a$.

%
%
\subsection{Structure of \asafe-floor quotient initial intervals and almost complementation}\label{subsec:12} 

The {\em $n$-floor reciprocal map} $\J_n(k) = \floorfrac{n}{k}: \sA[1, n] \to \sA[1,n]$ acting on the additive
partial order interval $\sA[1,n] = \{1, \ldots, n\}$, leaves both
the partial order intervals  $\ADset_1[1,n]$ and $\sD[1,n]$ invariant, and it acts as an involution on $\ADset_1[1,n]$
and $\sD[1,n]$. 
Under this map $\J_n$ the  initial intervals  $\ADset_1[1,n]$
of the  $1$-floor quotient order have an {\em almost complementation} property, described in Section \ref{subsec:23a}. 
Namely, we have $\ADset_1[1,n]= \ADsetsmall(n) \cup \ADsetlarge(n)$
where the two sets of ``small''  elements $\ADsetsmall(n)$ and ``large'' elements $\ADsetlarge(n)$
have at most one element in common, and if there is a common  element, it  is $s = \floor{\sqrt{n}}$.
The almost complementation property states that: the two sets $\ADsetlarge(n)$ and  $\ADsetsmall (n)$ are the same size, with $\J_n$
giving a bijective map between them, see Section \ref{subsec:23b}. 

We can define  analogues of the almost complementing sets:
for $a\ge 1$ we define  
\begin{equation*} 
\ADasetlarge(n) := \ADsetlarge(n) \cap \ADaset[1,n]
 \quad \mbox{and} \quad
\ADasetsmall(n) := \ADsetsmall(n) \cap \ADaset[1,n]. 
\end{equation*}
Using this definition, 
for each fixed $a \ge 2$ we find that the almost complementation property breaks down for $a$-floor quotients
on initial  intervals $\ADaset[1,n]$, for sufficiently large $n$.  
That is, for each $a \ge 2$, the  $n$-floor reciprocal map $\J_n$ does not preserve the interval $\ADaset[1,n]$. 
We show that the map $\J_n$ always gives an injection $\J_n: \ADasetlarge(n) \to \ADasetsmall(n)$,
so that in cardinality $\verts{ \ADasetlarge(n) } \le \verts{ \ADasetsmall(n) }$.
For $ a \ge 2$, this inequality is generally strict, as shown by example for $a=2$ in Figure \ref{fig:51} in Section \ref{subsec:52b}. 
We conclude there is an asymmetry for $a \ge 2$, that $\ADasetlarge(n)$ is  in general smaller than $\ADasetsmall(n)$, 
and  we give a heuristic to suggest a quantitative version of this asymmetry.

We study  the sizes of  $\ADasetsmall(n)$ and $\ADasetlarge(n)$.
In Section~\ref{subsec:51b} we prove that $\verts{ \ADasetsmall(n) } \ge \floor{\sqrt{\frac{n}{a}}}$ and that
 $\verts{ \ADset_{a} [1, n] } > \sqrt{\frac{n}{a}}$ for all $n > 1$.
 In  Section~\ref{subsec:52b}  we formulate a heuristic probabilistic argument suggesting that 
   $\ADaset^{-}(n)$ should have expected size about ${(\frac{2}{\sqrt{a}} - \frac{1}{a}) \sqrt{n}}$
 while $\ADasetlarge(n)$ should have expected size  about $\frac{1}{a} \sqrt{n}$.  
The heuristic predicts $\verts{ \ADaset [1, n]}$ should have expected size about 
$ 2 \sqrt{\frac{n}{a}}$. 
Numerical data for $a = 2$ for small $n$ is given which supports these predictions,
 and this data exhibits significant fluctuations of $\verts{\ADaset[1,n]}$ around any expected size function.
The heuristic argument given in Section~\ref{subsec:52b}, if interpreted as an expected value estimate, suggests that
\[
	\liminf_{n \to \infty}  \frac{ \verts{ \ADasetlarge (n) }}{\verts{ \ADasetsmall (n) }} \le \frac{1}{2\sqrt{a} - 1},
\]
may hold for fixed $a \ge 1$.

In Section~\ref{sec:7} we prove an averaged version of this heuristic. 

\begin{thm}[Average value of $a$-floor initial interval, variable $n$]\label{thm:16aa}
For all integers $a \ge 1$ and all $x \ge 1$,
\begin{equation} 
\label{eqn:intro-averaged}
	\frac{1}{x} \sum_{n=1}^x \verts{ \ADaset[1,n] } = \frac{4}{3} \sqrt{\frac{x}{a}} + O \left( \log x + a \right),
\end{equation}
in which the $O$-constant is independent of $a$ and $x$.
\end{thm}
Theorem~\ref{thm:16aa} is proved, relabeled  as Theorem~\ref{thm:52}, in Section~\ref{sec:7}.

The proof of Theorem \ref{thm:52} makes use of a classification of integers which are $a$-floor multiples of a given $d$, 
given in Section \ref{sec:a-floor-multiples}. 
The set of $a$-floor multiples $\ADmult_A(d)$ of $d$ is  
given by$\ADmult_a(d) := \{ n \in \NNplus : \, d \ADa n\}$.

%
%
%
\begin{thm}[Numerical semigroup of $a$-floor multiples]
\label{thm:a-floor-multiple-struct-0}
The set $\ADmult_a(d)$ of $a$-floor multiples of $d$
is a numerical semigroup.
It has the following properties.
\begin{enumerate}[(1)]
\item
The largest integer not in $\ADmult_a(d)$, its Frobenius number,  is 
$(d-1)(ad+1)$.

\item 
The number of positive integers not in $\ADmult_a(d)$ is $\frac12(d-1)(ad+2)$.

\item 
The minimal generating set of $\ADmult_a(d)$ is
$
\angles{ \gamma_0, \gamma_1, \ldots, \gamma_{d - 1} }
$
where 
\[
	\gamma_j = (ja + 1) d  + j = j(ad+1) + d \quad \mbox{for} \quad 0 \le j \le d - 1.
\]
\end{enumerate}
\end{thm}
Theorem \ref{thm:a-floor-multiple-struct-0} is proved as  Theorem \ref{thm:a-floor-multiple-struct} in Section \ref{sec:a-floor-multiples}.
Its proof generalizes  the $a=1$ case proved in \cite[Theorem 4.2]{LagR:23a}.
 

%
%
\subsection{Prior work}\label{subsec:14} 

The study of the $1$-floor quotient relation in \cite{LagR:23a} was motivated  by 
work of  J.-P. Cardinal \cite{Cardinal:10}.
The main topic of Cardinal's paper is
the construction
 of a commutative  quotient algebra $\mathbf{\sA}= \mathbf{\sA}_n$ of integer matrices, 
 for each $n \ge 1$,  
whose  rows and columns  are indexed by the
floor quotients of $n$, which form  the initial interval $\ADset_1[1,n]$ of the $1$-floor quotient partial order. 
Cardinal~\cite{Cardinal:10} showed that certain matrices  $\sM = \sM_n$ of the algebra $\sA_n$ have explicit connections with the Mertens function  
$M(n) = \sum_{j=1}^n \mu(j)$ 
and he related behavior of the norm of his matrices to the Riemann hypothesis \cite[Theorem 24]{Cardinal:10},
see also \cite{CardinalO:20}.


Cardinal proved the commutativity of the algebra $\mathbf{\sA}_n$ using a floor function identity which states the commutativity of certain dilated floor functions acting on the real line~\cite[Lemma 6]{Cardinal:10}:
for positive integers $k, \ell $ one has 
the identities
\begin{equation}
\label{eq:floor-dilation-commute}
\floor{ \frac{1}{k} \floor{ \frac{n}{\ell} }} 
= \floor{ \frac{1}{\ell} \floor{ \frac{n}{k} }} 
= \floor{ \frac{n}{k\ell} }.
\end{equation}
These identities are the basis of the existence of the  floor quotient partial order; they imply  the transitivity property of the floor quotient partial order. 
In  \cite{LMR:16}, the authors with T.  Murayama  classified  pairs
of dilated floor functions $(\floor{\alpha x} , \floor{\beta x})$
that commute under composition as functions on $\RR$.
They showed  that
the identities \eqref{eq:floor-dilation-commute} are  quite special: 
they are the only nontrivial cases
where the commutativity property of dilated floor functions occurs.
(The trivial cases are
where $\alpha=0$, or $\beta=0$ or $\alpha=\beta$.)

The existence of Cardinal's algebra seems related to the existence of the complementation structure in
the $1$-floor quotient partial order. We do not know of  an analogue of his algebra for the $a$-floor quotient
order for $a \ge 2$.

%
%
\subsection{Contents of paper}
\label{subsec:15}

 This paper has two main observations.

(a) As   $a \to \infty$  the $a$-floor quotient  orders converge monotonically pointwise  (in a precise sense) to 
the divisor order on $\NNplus$. We obtain quantitative information on the rate of convergence.

(b) The $1$-floor quotient order $\ADset_1$ and the divisor order $\sD$ have an almost complementation symmetry on initial intervals of their respective orders, 
but the $a$-floor quotient orders $\ADset_a$ for $a \ge 2$ break this almost complementation symmetry.  
We obtain quantitative information on  the amount of asymmetry. 

A more detailed description follows.

\begin{enumerate}
\item
Section \ref{sec:2}  recalls results on the $1$-floor quotient order and the divisibility order on $\NNplus$, which is a distributive lattice.
The divisibility order has a dilation symmetry which is broken  for the $1$-floor quotient order, 
whose breaking leads to the existence of the $a$-floor quotient orders  for $a \ge 2$.

\item
Section \ref{sec:a-floor-orders}  gives  six equivalent characterizations of the $a$-floor quotient relation in Theorem \ref{thm:a-equiv-properties}.
Section \ref{subsec:33} shows that  if $a < b$ in the additive order, then the order $(\NNplus, \ADa)$ strictly refines $(\NNplus, \ADof{b})$.

\item
Section~\ref{sec:4} studies the variation of the structure of $a$-floor quotient relations, as $a \geq 1$ varies.
We describe the set of $a$ for which a given pair $(d,n)$ satisfies $d \ADa n$, 
showing that  this set is an initial interval in the additive order.  
We determine a quantitative bound on the rate at which  the $a$-floor quotient orders converge to the divisor order as $a \to \infty$ on $\{1, \ldots, n\}$, see Theorem \ref{thm:interval-hierarchy-1}.
We introduce the {\em stabilization index} of an interval $\ADset_1[d,n]$ as the least value of $a$ where the ($\!\ADa\!$)-interval is order-isomorphic to $\sD[d,n]$.

\item
Section \ref{sec:a-duality} studies  the size and structure of the initial intervals $\ADaset[1,n]$ 
for fixed $a$. We give a decomposition $\ADaset[1,n] = \ADasetsmall(n) \cup \ADasetsmall(n)$.
The $n$-floor reciprocal map
$\J_n: \ADasetlarge(n) \to \ADasetsmall(n)$  is well-defined and injective.
In particular, $\verts{ \ADasetlarge(n) } \le \verts{ \ADasetsmall(n) }$
and we show by example that this inequality is generally strict for $a \ge 2$.
We show  $\verts{ \ADaset[1,n] } >  \sqrt{\frac{n}{a}}$. 
We present data for $a = 2$ showing that the values $\verts{ \ADaset[1,n] }$ are
not monotone in $n$ and show considerable fluctuation.
We formulate a heuristic argument suggesting that (on average) $\verts{ \ADaset[1,n] }$ should
be of size about $ 2 \sqrt{\frac{n}{a}}$. 

\item
Section \ref{sec:a-floor-multiples}  shows in Theorem \ref{thm:a-floor-multiple-struct} 
that the  set of  $a$-floor multiples of a given integer $d$
forms a numerical semigroup, and  characterizes the resulting class of numerical semigroups. 

\item
 Section \ref{sec:7} studies the average size of $\verts{ \ADaset[1,n] } $ for $1 \le n \le x$.
It shows this average size is asymptotic to 
$ \frac{4}{3} \sqrt{\frac{x}{a}}$ as $x \to \infty$. 
This answer is compatible with the heuristic for the size $\verts{ \ADaset [1,n] }$ of expected size $2 \sqrt{\frac{n}{a}}$ given in Section \ref{sec:a-duality}.
\item
 Section \ref{sec:concluding}
 addresses areas for further work. 
\begin{enumerate}[(i)]
\item Study of  the \Mobius{}
function of the $a$-floor quotient order; 

\item  Structure of the incidence algebra of the $a$-floor quotient poset; 

\item  Possible existence of a family of approximate divisor orders interpolating between the $1$-floor quotient order and the additive total order on $\NNplus$.
\end{enumerate} 

\end{enumerate}

%
%
\section{Preliminary results on 1-floor quotient order}
\label{sec:2} 

We review properties of the $1$-floor quotient order and its relation to the divisibility partial order,
shown in \cite{LagR:23a}.
A general reference for facts on partial orders is Stanley~\cite{Stanley:12}*{Chapter 3}.

%
%
\subsection{The 1-floor quotient order compared with divisibility order}
\label{subsec:21} 

The {\em divisibility partial order} $\sD := (\NNplus, \;\mid\;)$ is a distributive lattice in the sense of Birkhoff. 
That is, it  has a  well-defined join function $\vee$ (least common multiple), 
and a well-defined meet function $\wedge$ (greatest common divisor),
which obey the distributive law
$x \wedge (y \vee z)= (x \vee y) \wedge (x \vee z)$.
It also has a well-defined rank function with the rank of an element $d$ is the length of a minimal set of covering relations to the minimal element $1$; 
the element $1$ is assigned rank $0$.
We let $\Dset[d, n]$ denote an interval of the divisor order, i.e.
\[
\Dset[d, n] := \{ e \in \NNplus :  \,\, d \divides  e \,\, \mbox{and} \,\, e \divides n\}. 
\]
Each interval of the divisibility partial order with the induced order relation is itself a ranked distributive lattice.

In contrast, it was shown by examining various intervals in \cite[Sect. 2]{LagR:23a} that the  $1$-floor quotient poset $\ADset_1  = (\NNplus, \AD)$
does not have well-defined join (least upper bound)
or meet functions (greatest lower bound), so the $1$-floor quotient partial order on $\NNplus$ is not a lattice. 
In addition, it does not have a well-defined rank function. 

There exist similar examples for the $2$-floor quotient relation on $\NNplus$, showing
that it is not a lattice and does not have a well-defined rank function. 
We expect the same features hold for the $a$-floor quotient relation for each $a \ge 1$.

%
%
\subsection{Scaling-invariance properties: divisor and 1-floor quotient order}\label{subsec:22} 

The finite intervals $\sD[d,n]$ of the divisor poset $\sD = (\NNplus, \divides)$ have  a scaling-invariance property:
for each pair $(d, n)$, and each $a \geq 1$,  there is an isomorphism of  partial order intervals
\begin{equation}
\sD[d, n] \simeq \sD[ad, an] . 
\end{equation}
For  an interval $I= \sD[d,n]$ with endpoints $(d,n)$ we  call  $w(I) = w(\sD[d,n]) = \frac{n}{d}$ the {\em width } of the interval $I$.
The width is a scale-invariant quantity for the divisor poset: 
any  two intervals having the same width are order-isomorphic.
In particular, intervals with the same width have the same cardinality. 
(If $d \nmid n$ the cardinality is zero.) 

The  nonempty intervals $\ADset_1[d, n]$ of the  $1$-floor quotient order usually break scale-invariance. 
That is, given $d \AD n$, the order interval $\ADset_1[d,n]$ is generally not order-isomorphic to $\ADset_1[ad, an]$;
in particular, the  cardinality of 
$\ADset_1[ad, an]$ varies with $a \ge 1$.  
If $d \nmid n$ while  $ d \AD n$, then $\ADset_1[d, n]$ is nonempty, while  for large enough $a$, we have  $\ADset_1[ad, an] = \emptyset$. 
(See Example~\ref{exa:47} for variation with $a$.)


%
%
\subsection{Self-duality and almost  complementation properties: divisor order}
\label{subsec:23a}

The divisor poset $\sD = (\NNplus, \divides)$ has two important symmetry properties,
a self-duality property and an (almost) complementation property on each of its finite  intervals. 
We discuss it here for the initial interval $\sD[1,n]$.

\begin{enumerate}
\item[(1)] 
The  divisor poset  has  on each of its 
initial divisor intervals $\Dset[1,n]$  a {\em self-duality property},     
stating that the divisor partial order on $\Dset[1,n]$ is isomorphic to the dual
divisor partial order $\mid^{\ast} $ on $\Dset[1,n]$ given by
\[ 
d_1 \mid^{\ast} d_2\quad  \leftrightarrow \quad d_2 \mid d_1.
\]
That is:
\[
(\Dset[1,n], \divides) \simeq (\Dset[1,n], \;\mid^{\ast}\; ).
\]
This isomorphism  is established using  the {\em $n$-reciprocal map}   $\J_n: \Dset[1,n] \to \Dset[1,n]$ given by
$\J_n(d) = \frac{n}{d}$.
This map  is   {\em order-reversing}:
If $d \divides  e$ in the divisor order on $\Dset[1,n]$ then $ \frac{n}{e} \divides  \frac{n}{d}$. 
The map $\J_n$ is also an {\em  involution}:   $\J_n^{\circ 2}(d)=d$ on the domain $\Dset(1,n)$,
hence  it is a bijection.

\item[(2)]
The divisor partial order $\sD$ has an {\em almost complementation} property
of each interval $\Dset[1,n]$, with respect to the involution $\J_n$.  
It divides the order interval $\Dset[1,n]$  in two almost complementary parts
\[
\Dset[1,n] = \Dset^{-}(n) \bigcup \Dset^{+}(n),
\]
in which
\[
\Dset^{-}(n) = \{ d \in \Dset[1,n]: 1 \le d \le \sqrt{n} \}
\]
and 
\[
\Dset^{+}(n) :=  \{ d \in \Dset[1,n] :  1 \leq \J_n(d) \leq \sqrt{n} \}
\]
It is clear that 
$\J_n (\Dset^{+}(n) ) = \Dset^{-}(n)$
 and $\J_n (\Dset^{-}(n))  = \Dset^{+}(n)$.

We have
\[
\Dset^{+}(n) \bigcap \Dset^{-}(n) 
= \begin{cases}
\{s\} &  \mbox{if} \,\, n = s^2 \\
\emptyset & \mbox{if} \,\, n \ne s^2  \quad \mbox{for any} \,\, s \ge 1.
\end{cases} 
\]
If  $n = s^2$ then the common element $s$ is the unique fixed point 
of the involution $\J_n$. 
\end{enumerate}
The ``almost complementation'' property is  different from the notion of  a complementation on  the (Birkhoff) lattice
$\sD[1,n]$, which requires that each element $a$ of the interval has a complement $a'$ whose join with $a$ is the maximal element $n$
and whose meet is the minimal element $1$. A divisor order interval $\sD[d,n]$ has a complement in the (Birkhoff)  lattice sense if and only if $\frac{n}{d}$ is squarefree.

%
%
\subsection{Self-duality and almost complementation properties: 1-floor quotient order}
\label{subsec:23b}

The $1$-floor quotient order retains  weaker forms  of self-duality and almost complementation properties,
with respect to the {\em $n$-floor reciprocal map} $\J_n: \ADset_1[1,n] \to \ADset_1[1,n]$,
given by 
\begin{equation}
\J_n (k ) = \floor{ \frac{n}{k} }.
\end{equation} 
The map $\J_n$ acts as an involution  on the set $\ADset_1[1,n]$.
The restriction of $\J_n$ to the smaller domain $\sD[1,n]$ of divisors of $n$ agrees with the  $n$-reciprocal map  $\J_n$ defined in Section~\ref{subsec:22}, 
since $\floorfrac{n}{d} = \frac{n}{d}$ for divisors $d$ of $n$.
Moreover, $\ADset_1[1,n]$ is the maximal subset of $\{1,2,\ldots,n\}$ on which $\J_n$ acts as an involution (\cite[Lemma 4.1]{LagR:23a}).

The involution $\J_n$ acting on the initial interval $\ADset_1[1,n]$ of the  floor quotient order is not order-reversing,
unlike its action on the divisor order interval $\sD[1,n]$. 
The map $\J_n$ has instead the  weaker property that:
the image under $\J_n$ of two comparable elements $d \AD e$  in the floor quotient order either reverses their order $\J_n(e) \AD \J_n(d)$ or else gives two incomparable elements, which we write
$\J_n(e) \perp_{1} \J_n(d)$.  
The incomparable case can only occur when $d \le  \sqrt{n}$.

The following result summarizes  the almost complementation action on $1$-floor quotients. 
This result  is due to Cardinal~\cite[Proposition 4]{Cardinal:10}
(stated in a different notation). 
The statement below combines \cite[Lemma 4.1]{LagR:23a}  and \cite[Proposition 4.2]{LagR:23a}.

\begin{prop}
[Almost complementation for initial $1$-floor quotient intervals]\label{prop:interval-complement}
The $n$-floor reciprocal map $\J_n$  acts as an involution when  restricted to the domain
$\ADset_1[1,n]$ of all floor quotients of $n$, i.e.
\begin{equation}\label{eq:37a}
\J_n^{\circ 2} (d) = d \quad \mbox{for all} \quad  d \in \ADset_1[1,n].
\end{equation} 
Given  an integer $n \geq 1$,
let 
\begin{equation*}
\ADsetsmall(n) = \{d : d \AD n,\, d \leq \sqrt{n} \}
\qquad\text{and}\qquad
\ADsetlarge(n) = \{d : d \AD n,\, \floorfrac{n}{d} \leq \sqrt{n} \}.
\end{equation*}
\begin{enumerate}
\item 
We have
$ 
\ADset[1,n] = \ADsetsmall (n) \bigcup \ADsetlarge(n).
$
The  map $\J_n: k \mapsto \floor{n/k}$  restricts to a bijection sending $\ADsetlarge(n)$ to
$\ADsetsmall(n)$ 
and to a bijection sending $\ADsetsmall(n)$ to
$\ADsetlarge(n)$. 

\item 
Moreover, letting $s= \floor{\sqrt{n}}$,  one has
\begin{equation}
\ADsetsmall(n) = \{ 1, 2, \ldots, s\},
\qquad\text{and}\qquad
\ADsetlarge(n) 
= \{ n,  \floor{\frac{n}{2} }, \floor{\frac{n}{3} }, \ldots, 
\floor{ \frac{n}{s} } \}.
\end{equation}
The intersection of  $\ADsetsmall(n)$ and $\ADsetlarge(n)$ is
\begin{equation*}
\ADsetsmall(n) \bigcap \ADsetlarge(n) = \begin{cases}
\{s\} &\text{if }n < s(s+1) \\
\emptyset &\text{if }n \geq s(s+1).
\end{cases}
\end{equation*}
\end{enumerate}
\end{prop}


\begin{rmk}
(1)   Part (2) of Proposition~\ref{prop:interval-complement}
 determines the cardinality of  $|\ADset_1[1,n]|$  as follows.
 Set $s = \floor{\sqrt{n}}$.
\begin{enumerate}
\item[(i)] If $s^2 \leq n < s(s+1)$, then  $ \floor{ \frac{n}{s} }=s$
and  
 $|\ADset_1[1,n]|=2s-1$.
The element $s \in \ADset_1[1,n]$
is the unique fixed point of the map  $\J_n$ on $\ADset_1[1,n]$.

\item[(ii)] If $ s(s+1) \leq n < (s+1)^2 $, then $ \floor{ \frac{n}{s} } = s+1$ and   $|\ADset_1[1,n]| = 2s$.
The  map $\J_n$ has no fixed points on $\ADset_1[1,n]$.
\end{enumerate}

(2) The almost complementation property of $1$-floor quotients in Proposition~\ref{prop:interval-complement} fails to  hold for $a$-floor quotients with $a \ge 2$,
see Section \ref{sec:a-duality}.

\end{rmk}


%
%
\section{The family of \asafe-floor quotient orders}
\label{sec:a-floor-orders}
\setcounter{equation}{0}

In this section we  prove six equivalent characterizations of  the $a$-floor quotient relation, given  in Theorem \ref{thm:a-equiv-properties}.
We  deduce Theorem \ref{thm:a-approx-order}, showing that for each $a$, the  $a$-floor quotient relation is  a partial order that is an  approximate divisor order. 
We then show that the  $a$-floor  order   refines the $b$-floor order whenever $a \le b$, 
and that as $b \to \infty$ the order $(\NNplus, \fquo{b})$ approaches the divisor order on $\NNplus$ (Theorem \ref{thm:hierarchy-11}).
It follows that the  $1$-floor quotient  relation  is maximally refined 
among all these partial orders.  
(That is, it has the largest number of incidences in its order relation).

%
%
\subsection{Characterizations of \asafe-floor quotients}
\label{subsec:31a}

  The $a$-floor quotient has a  set of equivalent characterizations, which parallel
  and generalize those for the $1$-floor quotient given in \cite[Theorem 3.1]{LagR:23a}.
   
%
%
\begin{thm}
[Characterizations of  $a$-floor quotients] \label{thm:a-equiv-properties}  
For  a fixed integer $a \ge 1$, the following conditions on positive integers $\dd,n$ are equivalent
to $d \ADa n$:
\begin{enumerate}
\item {\em ($a$-Cutting property)} 
There exists a positive integer $k$ such that 
$ \displaystyle a\dd = \floor{ \frac{an}{k} }$.

\item {\em ($a$-Covering property)}
There exists a positive integer $k$ such that
\[
k [ \dd , \dd + 1/a) \supseteq [n, n + 1/a), 
\]
where both sides denote half-open intervals in $\RR$.

\item  {\em ($a$-Intersection property)}
There exists a positive integer $k$ such that
$$k [ \dd , \dd + 1/a) \cap [n, n + 1/a) \neq \emptyset.$$

\item  {\em ($a$-Strong remainder property)}
There is a positive integer $k$ such that
\[
n = \dd k + {r} \quad\text{where the remainder $r$ satisfies }\, \, 0 \leq  r < \min\left( \dd,\frac{k}{a} \right).
\]

\item {\em ($a$-Tipping-point property)}
There holds the strict inequality 
$
\displaystyle
\floorfrac{n}{\dd} > \floorfrac{n}{\dd + {1}/{a} }
$.

\item {\em ($a$-Reciprocal-duality property)} 
There holds the identity
\begin{equation}
a\dd = \floor{ \frac{an}{\floor{n/\dd}} }.
\end{equation}
%
%
Equivalently
\begin{equation}
\dd \leq \frac{n}{\floor{n/\dd}} < \dd + \frac{1}{a}.
\end{equation}
\end{enumerate}
\end{thm}

%
%
\begin{proof}[Proof of Theorem \ref{thm:a-equiv-properties}]

$(1) \Leftrightarrow (2).$ 
Suppose  $ad= \floor{ \frac{an}{k} }$ for some $k$, which is property (1).
Then $an= k(ad) +r $ with $0 \le r \le k-1$,
so
 $[an, an+1) \subset [kad, kad+ k) =k[ad, ad+1)$.
Dividing by $a$, we obtain
$$\textstyle [n, n+\frac{1}{a})  \subset k[d, d+ \frac{1}{a} ).$$   
This is the $a$-covering property (2) for $(d,n)$.
These steps are all reversible, 
so the properties are equivalent.


$(2) \Leftrightarrow (3).$ 
The direction $(2) \Rightarrow (3)$ is immediate. For 
$(3) \Rightarrow (2)$ if $k[d , d+\frac{1}{a})$ contains some $x \in [n,n+\frac{1}{a})$ then,
multiplying by $a$, we have $k[ad, ad+1)$ contains some $y=ax \in [an, an+1)$. 
But then contains the full half-open interval $[an, an+1)$, 
because the interval $k[ad, ad+1)$ has integer endpoints.

$(1) \Leftrightarrow (5)$.
For positive integers $\dd,n$, we define the {\em $a$-cutting length set} to be
 the set of integers
\begin{equation}
\sK_{a}(\dd,n) := \{ k \in \NNplus : a\dd = \floor{\frac{an}{k} } \}.
\end{equation}
Property (1) says that this set is nonempty.
We have
\[
\sK_a(\dd,n) 
= \{ k\in \NNplus : a\dd \leq \frac{an}{k} < a\dd + 1\}
= \{ k \in \NNplus : \dd \leq \frac{n}{k}  < \dd + \frac{1}{a}\} .
\]
The lower bound $\dd \leq \frac{n}{k} $
is equivalent to 
$k \leq \floor{ \frac{n}{\dd} }$.
The upper bound $\frac{n}{k} < \dd + \frac{1}{a}$
is equivalent to
$\floorfrac{n}{\dd +{1}/{a}}  < k$.
Therefore
\begin{equation} 
\label{eqn:Kad-bound} 
\sK_a(\dd,n) = \{ k \in \NNplus : \floor{ \frac{n}{\dd +1/a}} < k \leq \floor{ \frac{n}{\dd} } \},
\end{equation} 
so $\sK_{a}{n}$ is nonempty if and only if $\floor{ \frac{n}{\dd+\frac{1}{a}} } < \floor{ \frac{n}{\dd} }$,
which is the tipping point property (5).

$(1) \Leftrightarrow (6)$.
The direction $(6) \Rightarrow (1)$ is clear.
Now suppose (1) holds. The proof  of  $(1) \Leftrightarrow (5)$
established  \eqref{eqn:Kad-bound}, showing that 
\[
ad = \floorfrac{an}{k} \quad\text{is equivalent to }\quad 
\floorfrac{n}{d + 1/a} < k \leq \floor{ \frac{n}{d} } .
\]
Property (1) says that  
$\sK_a(\dd,n) = \{ k \in \NNplus : a\dd = \floorfrac{an}{k} \}$
is nonempty. 
Now any  $k \in \sK_{a}(\dd, n)$ must have $\frac{an}{k} \ge ad$, 
or equivalently $\frac{an}{ad} \ge k$. 
Thus if  $\sK_{a}(\dd, n)$  is nonempty then it must contain $k = \floorfrac{an}{ad}$,
which is the involution-duality property (6).

$(4) \Leftrightarrow (6)$.
Suppose property (4) holds, so there is a positive integer $k$ such that
$
n = \dd k + r$ where 
$0 \leq r < \min(\dd, k/a).$
Then
$
\floorfrac{n}{d}
= \floor{k + \frac{r}{d}}
= k 
$,
hence
\[
\floorfrac{an}{ \floor{n/d}}
= \floorfrac{adk + ar}{k}
= \floor{ ad + \frac{ar}{k} }
= ad
\]
which is the identity in property (6).

It remains to show that if property (4) does not hold for $k$,
then property (6) does not hold. 
For a given $(d,n)$, dividing $n$ by $d$ with remainder gives  
$n = dk + r$
for integers $k,r$ with  $0 \leq r < d$
and  $k = \floor{n/d}$.
Property (4) does not hold exactly when the remainder $r$
satisfies $r \geq k/a$.
In that case  
\[
\frac{an}{ \floor{n/d}}
= \frac{adk + ar}{k}
= ad + \frac{ar}{k} \geq ad + 1,
\]
which implies
$\floorfrac{an}{\floor{n/d}} > ad,$
so that property (6) does not hold. 
\end{proof} 

We can now prove that each $a$-floor quotient relation is a
partial order that is an approximate divisor order.

\begin{proof}[Proof of Theorem~\ref{thm:a-approx-order}]
We show all $a$-floor quotient orders are  partial orders.
This  was already shown for $a=1$ in \cite{LagR:23a}. 
Assuming the property for $a=1$,  the partial order property is inherited
for general $a \ge 2$  by 
the definition that $d \ADa n $  if and only if $ad \AD an$. 
The restriction of the partial order
$\AD$ to the subset $a\NNplus$ is automatically a partial order.

Alternatively, the partial order properties of reflexivity, symmetry and transitivity  
can be deduced directly from  characterizations given in  Theorem \ref{thm:a-equiv-properties} above.

If $d \mid n$, then $ad \mid an$, so $ad \AD an$, which by definition implies $a \ADa n$.
Thus $\ADa$ refines the divisor order. 

If $d \ADa n $,
then the condition in Theorem~\ref{thm:a-equiv-properties} (1) implies that $d \leq n$ in the additive order.
Hence $\ADa$ is refined by the additive total order.
We conclude that $\ADset_a = (\NNplus, \ADa)$ is an approximate divisor order, for any $a \geq 1$.
\end{proof} 

%
%
\subsection{Hierarchy of \asafe-floor quotient orders: Proof of Theorem \ref{thm:hierarchy-00}}\label{subsec:33}  

We derive Theorem \ref{thm:hierarchy-00} from the following result,
which reformulates its  condition (2)
in a more convenient form.

%

\begin{thm}[$a$-floor quotient order hierarchy]
\label{thm:hierarchy-11}  
 The $a$-floor quotient partial orders $\ADset_a$ have the following properties.
 
(1) If $a < b$ (in the additive order) then the $a$-floor quotient partial order $\ADset_a$ strictly refines the
$b$-floor quotient partial order  $\ADset_b$.

(2) The intersection of the $a$-floor quotient partial orders $ \ADset_a$ over all  $a \ge 1$
is  the divisor partial order  $\sD = (\NNplus, \divides)$.
Namely,
$$
\bigcap_{a \geq 1} \{ (d,n) \in \NNplus \times \NNplus : d \ADa n\} = \{ (d,n) \in \NNplus \times \NNplus : d \divides n\}.
$$
\end{thm}

\begin{proof}

(1) 
Theorem \ref{thm:a-equiv-properties} (6) states that
\begin{equation}\label{eq:62}
d \fquo{a} n \qquad\text{if and only if}\qquad 
d \leq \frac{n}{\floor{n/d}} < d + \frac1a .
\end{equation}
If $d \fquo{b} n$ holds, 
it follows that $d \fquo{a} n$ 
for all integers $a \leq b$. 
Thus $\fquo{a}$ refines $\fquo{b}$.

It remains to show that for $a<b$ the refinement is strict, 
i.e. there exists a pair $(d,n)$ having
$d \fquo{a} n$ but $d \not\!\fquo{b} n$. 
A suitable choice is  
$ (d,n) = (2, 2b  +1) $, so $\floor{n/d} = b$. 
The condition \eqref{eq:62} becomes
\[
2 \ADa 2b + 1 \qquad\text{if and only if}\qquad
2 \leq \frac{2b+1}{b} = 2 + \frac{1}{b} < 2 + \frac{1}{a} . 
\]
Thus $2 \ADa 2b + 1$ assuming $a < b$, but $2 \not\!\fquo{b} 2b+1$.

(2) We show that
\begin{equation}\label{eqn:hierarchy}
\bigcap_{a \geq 1} \{ (d,n) \in \NNplus \times \NNplus : d \ADa n\} = \{ (d,n) \in \NNplus \times \NNplus : d \divides n\}.
\end{equation}
Since each $a$-floor quotient order $\ADa$ refines the divisor order,
we have the inclusion
\[
\bigcap_{a \geq 1} \{ (d,n) \in \NNplus \times \NNplus : d \ADa n\} \supseteq \{ (d,n) \in \NNplus \times \NNplus : d \divides n\}.
\] 
 On the other hand, if $d \ADa n$ for all $a\geq 1$,
 then condition \eqref{eq:62}
 implies that 
\[
 \frac{n}{\floor{n/d}} = d \qquad\Leftrightarrow\qquad
 n = d \floor{n/d},
\]
 so $d$ divides $n$.
This shows the reverse inclusion. 
\end{proof}

%
%
\section{Hierarchy of \asafe-floor quotient orders: Stabilization bounds}\label{sec:4}
 
 We study the effect of varying $a$ on intervals of the $a$-floor quotient order. 
 
%
%
\subsection{\asafe-floor scaling sets}
\label{sec:52} 

%
\begin{defi}\label{def:51}
For a pair $(d,n)$ of positive integers, the {\em  $a$-floor scaling set} $\sS(d,n)$ is
the set of $a$ for which $d$ is an $a$-floor quotient of $n$, 
i.e.
\[
\sS(d, n) := \{ a\in \NNplus  : \,\, d \fquo{a} n  \} .
\]
\end{defi}

The $a$-floor scaling set $\sS(d,n)$ may be infinite, finite, or empty.
%
\begin{exa}\label{exa:51} 
We have
\[\sS (5, 15)=\NNplus, \quad \sS( 5, 16) = \{ 1, 2 \},\quad \sS(5, 17) = \{ 1\},\quad  \sS(5, 18) = \emptyset.\]
\end{exa} 
%
\begin{defi}
Given positive integers $d$ and $n$,
the {\em quotient discrepancy} $\delta(d,n)$ is
\begin{equation}
\label{eq:discrepancy}
	\delta(d,n) = n - d \floor{\frac{n}{d}} .
\end{equation}
\end{defi}
Note that $\delta(d,n)$ always satisfies $0 \leq \delta(d,n) < d$. 
We have $\delta(d,n) = 0$ if and only if $d$ divides $n$,
and $\delta(d,n) < \floor{n/d}$ if and only if $d$ is a floor quotient of $n$.
Moreover, if $d$ is a floor quotient of $n$, then
$\delta(d,n) = \delta(\floor{n/d}, n)$.

%
\begin{thm}[Structure of  $a$-floor scaling sets]
\label{thm:scaling-trichotomy}
For positive integers $d$ and $n$, the set $\sS(d,n)$ of all $a$ where $d \ADa n$ is either $\NNplus$, 
or an initial  interval $\sA[1, m]$ of $\NNplus$ of the additive order $\sA = (\NNplus, \leq)$, or the empty set. 
The following trichotomy holds.

\begin{enumerate}[(i)]
\item
If $d \mid n$, then $\sS (d, n) = \NNplus$.
\item
If $d \AD n$ and $d \nmid n$, then 
 $\sS(d, n) = \{ 1, 2, 3, \cdots, \tilde{a}(d, n)\}$,
where $\tilde{a}(d,n)$  is the largest integer such that
\[
	\tilde{a}(d, n) < \frac{ \floor{n/d} }{ \delta(d,n) }
\]
and $\delta(d,n)$ is defined by \eqref{eq:discrepancy}.

\item
If $d \notAD n$, then $\sS(d, n) = \emptyset$.
\end{enumerate} 
\end{thm}

\begin{proof}
Theorem~\ref{thm:hierarchy-00} (1) implies that $\sS(d,n)$ must be an initial interval in the additive order, 
i.e. if $b \in \sS(d,n)$ and $a < b$ then $a \in \sS(d,n)$.

(i) Suppose $d \mid n$. 
Then $d \ADa n$ for all $a \ge 1$, by Theorem~\ref{thm:a-approx-order}.
Therefore $\sS (d, n) = \NNplus$.

(ii) 
Suppose $d \AD n$ but $d \nmid n$. 
Using criterion \eqref{eq:62},
we have $a \in \sS(d,n)$ if and only if 
\begin{equation*}
\frac{n}{\floor{n/d}} < d + \frac1a 
\quad\Leftrightarrow\quad
a < \frac{ \floor{ {n}/{d} }}{n - \floor{ {n}/{d} } d} 
=: \frac{ \floor{n/d} }{\delta(d,n)} .
\end{equation*}
The quantity $\delta(d,n) = n - d\floor{n/d}$ is positive by the assumption that $d$ does not divide $n$.
In this case we define $\tilde{a}(d,n)$ to the largest integer that is strictly less than
$\frac{ \floor{n/d} }{ \delta(d,n)}$.

(iii)  
If $d\notAD n$,
then Theorem \ref{thm:hierarchy-11}
implies that $d \notADa n$ for all $a \geq 1$.
Therefore in this case $\sS(d, n) = \emptyset$.
\end{proof}

%
%
\subsection{Hierarchy  bounds for  stabilization to divisor order}\label{subsec:54}  

 %
\begin{thm}[$a$-floor quotient  order stabilization]
\label{thm:interval-hierarchy-1} 
The  $a$-floor quotient partial order  on the interval $\ADset_a[d,n]$, 
is identical  to the divisor partial order $\mid$ on the  interval $\sD[d,n]$,  
whenever  $a \ge \floorfrac{n}{d}$. 
That is,
\[
	(\ADaset [d,n], \ADa)  = ( \sD[d,n], \mid)  \qquad \text{for all}  \quad a \ge \floor{ \frac{n}{d}}.
\]
If in addition $d \mid n$,
then  the $a$-floor quotient partial order  on  $(\ADaset [d,n], \ADa)$ is identical  to 
the divisor partial order on  $(\sD[d,n], \mid)$,  for $a \geq \floorfrac{n}{d+1}$. 
\end{thm}  

\begin{proof}
The result follows using  Theorem \ref{thm:scaling-trichotomy} (ii). 
Since all  the order relations in $\sD[d,n]$ are  included
in  $\ADset_a[d,n]$ for each $a \geq 1$,  if $a$ is such that
 $\ADset_a[d, n]$ is not order-isomorphic to $\sD[d, n]$, then there
must exist $e, m \in \ADset_a[d,n]$ having an extra order relation
\[
	d \ADa e \ADa m \ADa n \quad \mbox{with} \quad e \nmid m.
\]
The assumption $e \nmid m$ implies that $\delta(e, m) = m - e \floor{m / e} \geq 1$.
Since  $a \in \sS(e, m)$, Theorem~\ref{thm:scaling-trichotomy} (ii) implies that necessarily
\begin{equation}\label{eqn:a-bound}
	a < \frac{\floor{m / e}}{\delta(e, m)} \leq \floorfrac{m}{e}.
\end{equation}
Since $m \leq n$ and $e \geq d$, this implies $a < \floorfrac{n}{d}$ as claimed in the first part of the theorem.


For the second part of the theorem, assuming $d \mid n$, we treat two cases.
Any pair $(e,m)$ with $e \nmid m$ cannot be $(d, n)$, so either $e \ge d+1$ and $m=n$, or else $e \ge d$ and $m \le \frac{n}{2}.$

{\bf Case 1.}
In the first case $(e,m) := (e,n)$ has $e \nmid n$.  Let 
$k = k_{min}(d,n) $ denote  the least  integer $k \ge d$ such that $k \nmid n$,
so that $e \ge k_{min}(d,n)$. 
Now \eqref{eqn:a-bound} gives, for $k= k_{min}(d,n)$, 
\[
a < \frac{\floorslash{n}{e}}{\delta(e, n)} \le \floorfrac{n}{e} \le \floorfrac{n}{k}.
\]
Thus  for all pairs  $(e,n)$ occurring  this subcase
 stabilization to the divisor order occurs for $a \ge \floorfrac{n}{k_{min}(d,n)}$ 
Since $k_{min}(d,n) \ge d + 1$, we obtain stabilization for 
these $(e,m)$ for all  $a \ge \floorfrac{n}{d+1}$.

{\bf Case 2.}
In the second case $(e, m)$ satisfies $e \nmid m$ with $m \le \frac{n}{2}.$ 
Now \eqref{eqn:a-bound} gives
\[
a < \frac{\floorslash{m}{e}}{\delta(e, m)} \le \floorfrac{m}{e} \le \floorfrac{n}{2e} \le \floorfrac{n}{2d}.
\]
Thus stabilization to the divisor order occurs for all  $(e,m)$ occurring in this subcase
 in the range $a \ge  \floorfrac{n}{2d}$.

We conclude that for $d \mid n$ 
we have stabilization for $a \ge \max( \floorfrac{n}{k_{min}(d, n)}, \floorfrac{n}{2d}).$ 
Since $k_{min}(n, d) \ge d+1$ hence $\floorfrac{n}{k_{min}} \leq \floorfrac{n}{d + 1}$, and  $\floorfrac{n}{2d}\le \floorfrac{n}{d+1}$, 
we obtain stabilization of $(\ADaset [d,n], \ADa)  = ( \sD[d,n], \mid)$
to the divisor order $( \sD[d,n], \mid)$   for $a \ge \floorfrac{n}{d+1}$.
\end{proof}

A weaker condition than equality-as-posets asks how large $a$ must be to guarantee an equality of the underlying
sets of $\ADset_a[d,n] $ and $\sD[d,n]$.

\begin{exa}\label{exa:68}
The equality as sets  of $\ADaset[1,n]$ and $\sD[1,n]$ can  occur 
without isomorphism of the order relations. 

For example, for $a = 5$
the interval $\ADset_5[1,26] = \{1,2,13,26\}$
is equal to $\sD[1,26]$ as a set.
However, the order relation $\fquo{5}$ restricts to a total order $1 \fquo{5} 2 \fquo{5} 13 \fquo{5} 26$,
whereas the divisor order on $\{1,2,13,26\}$ is not a total order. 
On this example Theorem \ref{thm:interval-hierarchy-1}  has $k_{min}=3$, $2d=2$, and asserts agreement with the divisor order
for $a \ge \floorfrac{26}{2} = 13.$
The stabilization with the divisor order in fact occurs at $a = 7$.
\end{exa}

 The following example gives intervals  $\ADset_{a}[1, 10]$  in the $a$-floor quotient order 
 for various $a$.
  Recall that
  $
  \ADset_{a}[d, n] := \frac{1}{a} \left( \ADset_1[ad, an] \bigcap a\NNplus\right).
  $

\begin{exa}\label{exa:56} 
The $a$-floor quotient interval $\ADaset[1,10]$  is shown in the table below, for ${a = 1, 2, 3, \ldots}$ where $\ADset_\infty := \sD$.
\begin{center}
\begin{tabular}{ccc}
 $a$ & elements in $\ADaset[1, 10]$ & size \\ \hline \\[-0.8em]
$1$ & $\{ 1, 2, 3, 5, 10\}$ & 5 \\
$2$ & $\{ 1, 2, 3, 5, 10\}$ & 5 \\
$3$ & $\{ 1, 2, 5, 10\}$ & 4 \\
$a \geq 4$ & $\{ 1, 2, 5, 10\}$ & 4
\end{tabular}
\end{center}
Theorem \ref{thm:interval-hierarchy-1} asserts  that stabilization to the divisor order
 occurs by $a= \floorfrac{n}{2}=5$. Stabilization  occurs at $a \ge 3$.
\end{exa} 

\begin{exa}
 The $a$-floor quotient interval $\ADaset[1,11]$ 
 is shown in the table below, for $a = 1,2,3$ and $\infty$ where $\ADset_\infty := \sD$.

\begin{center}
\begin{tabular}{ccc}
 $a$ & elements in $\ADaset[1, 11]$ & size \\ \hline \\[-0.8em]
$1$ & $\{ 1, 2, 3, 5, 11\}$ & 5 \\
$2$ & $\{ 1, 2,  11\}$ & 3 \\
$3$ & $\{ 1, 2, 11\}$ & 3 \\
$a \geq 5$ & $\{ 1, 11\}$ & 2
\end{tabular}
\end{center}
Theorem \ref{thm:interval-hierarchy-1} asserts  that stabilization to the divisor order
 occurs by $a=\floorfrac{n}{k}   = \floorfrac{11}{2} = 5$. 
 \end{exa} 
 
 The rescaled interval $a \ADaset[1,n]$ is always contained in the set $\ADset_1[ ad, an]$;
 compare   the Example \ref{exa:56} for $\ADset_a[1,10]$ with 
 the full interval $\ADset_1[ ad, an]$
 (from Example 5.2 of \cite{LagR:23a}), given below.

\begin{exa}\label{exa:47}
The floor quotient intervals $\ADset [a,10a]$ for various $a$ are shown in the table below.
\begin{center}
\begin{tabular}{ccc}
 $a$ & elements in $\ADset_1[a, 10a]$ & size \\ \hline \\[-0.8em]
$1$ & $\{ 1, 2, 3, 5, 10\}$ & 5 \\
$2$ & $\{ 2, 4, 5, 6, 10, 20\}$ & 6 \\
$3$ & $\{ 3, 6, 7, 10, 15, 30\}$ & 6 \\
$4$ & $\{ 4, 8, 13, 20, 40\}$ & 5 \\
$9$ & $\{ 9, 18, 45, 90\}$ & 4 \\
\end{tabular}

\end{center}
\end{exa}

\subsection{Improved bounds on hierarchy stabilization}\label{subsec:43}

The {\em set-stabilization threshold} of the integer $n$
is the smallest
value of $a$ such that the interval $\ADaset[1, n]$ is equal, as a set, to the interval $\sD[1, n]$.
The {\em poset-stabilization threshold} of $n$ is the smallest value of $a$ such that the interval $\ADaset[1, n]$, as a poset, is equal to the interval $\sD[1, n]$.

In general, the set-stabilization threshold  is smaller than the poset-stabilization threshold.
See Figure~\ref{fig:stabilization} for a comparison of the set- and poset-stabilization threshold  for $n = 1, 2, \ldots, 100$.
\begin{figure}[h!]
\centering
	\includegraphics[scale=0.51]{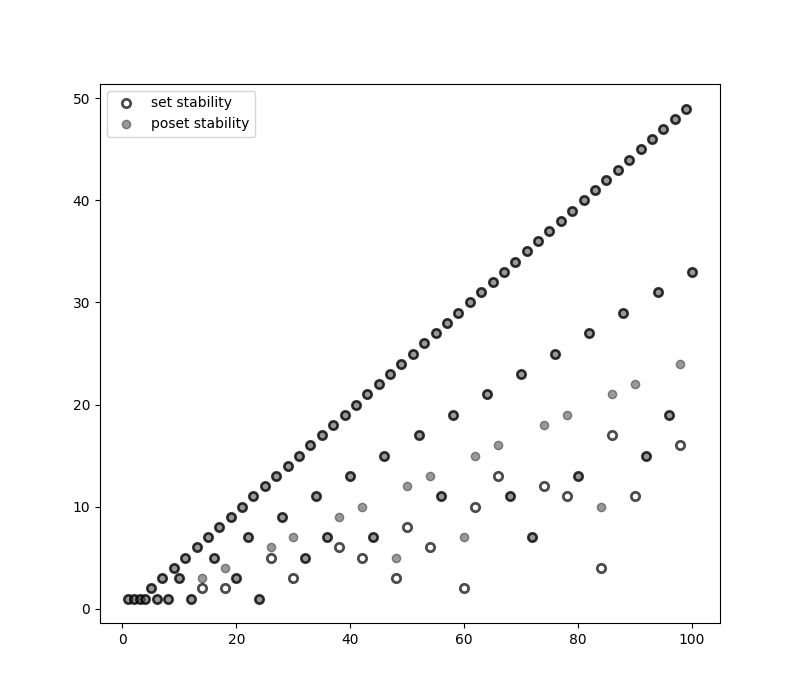}
\caption{Plot of set- and poset-stabilization thresholds for $\ADset_a[1,n]$ for $1 \leq n \leq 100$.
The horizontal axis is the $n$-variable and the vertical axis is the $a$-variable.}
\label{fig:stabilization}
\end{figure}

%
%
\section{Action of floor reciprocal map:
Broken reciprocal duality}\label{sec:a-duality}  
\setcounter{equation}{0}

The reciprocal-duality property for the $1$-floor quotient partial order 
on $\ADset_1[1,n]$ is weakened to the property
(6) given in Theorem \ref{thm:a-equiv-properties}. 
In particular,
the sets $\ADasetlarge (n)$ and $\ADasetsmall(n)$ are not in $1$-to-$1$ correspondence.
The reciprocal duality symmetry (at the level of sets)
on the  initial interval $\ADaset[1,n]$ for $n \ge 1$ is
broken for $a \ge 2$.  

We show the $n$-floor reciprocal map  $\J_n : \ADasetlarge(n) \to \ADasetsmall(n)$ is
injective, but in the reverse direction the map $\J_n$ only sends $\ADasetsmall(n)$ into $\ADsetlarge(n)$.
The image of $\ADasetsmall$ is always
contained in $\ADsetlarge$, but for $a \ge 2$,
is usually not in $\ADasetsmall(n)$.

%
%
\subsection{Asymmetry of initial \asafe-floor quotient intervals}\label{subsec:51b} 

We have $\ADaset[1, n] \subseteq \ADset_1[1,n]$,
which is proved in 
Theorem~\ref{thm:hierarchy-11}.
It guarantees that each  $d \in \ADaset[1,n]$ can be  written as $d = \floorfrac{n}{k}$
for a unique $k \in \ADset_1[1,n]$;  
however it need not hold that $k \in \ADaset[1,n]$.

Since the $a$-floor quotient interval  $\ADset_{a}[1,n]$ is contained in $ \ADset_{1}[1,n]$, 
we have a decomposition into ``small'' and ``large'' quotients
\[ 
\ADset_{a}[1,n]= 
\ADasetsmall(n) \bigcup \ADasetlarge(n),
\]
with 
\[
\ADasetsmall(n) :=  \ADaset[1,n] \bigcap \ADset_1^{-}(n)
= \{d \in \ADaset[1,n] : d \leq \sqrt{n} \} 
\]
and
\[
\ADasetlarge(n) :=  \ADaset[1,n] \bigcap\ADset_1^{+}(n)
= \{d \in \ADaset[1,n] : \floorfrac{n}{d} \leq \sqrt{n} \}  .
\]

\begin{lem}\label{lem:62b}
The function $\J_n$ maps the set $\ADasetlarge(n)$ injectively into the set $\ADasetsmall(n)$.
Consequently, for all $n \ge 1$, 
\begin{equation}\label{eqn:minus-beats-plus}
 \verts{ \ADasetlarge(n) } \le \verts{ \ADasetsmall(n) }.
 \end{equation}
\end{lem}
\begin{proof}
Suppose $d \in \ADaset[1,n]$. Then Property (4) of Theorem \ref{thm:a-equiv-properties} 
says this is equivalent to
\[
	n = \dd k + r \qquad 
	\text{where} \quad k = \J_n(d) = \floorfrac{n}{d}
	\quad\text{and}\quad 0 \leq r < \min(d, k/a).
\]
Note that $k = \J_n(d)$ and $d \AD n$ implies $d = \J_n(k)$, by Proposition~\ref{prop:interval-complement}.

If $d \in \ADasetlarge(n)$,
then $k \leq d$ 
so the bound $r < k/a$ also implies $r < d/a$.
Moreover, since $a \geq 1$ the bound $r < k/a $ implies $r < k$.
Thus
\[
n = kd + r \qquad
\text{where} \quad d = \J_n(k) = \floorfrac{n}{k}
\quad\text{and}\quad 0 \leq r < \min(k, d/a).
\]
By Property (4) of Theorem~\ref{thm:a-equiv-properties},
this implies $k = \J_n(d) \in \ADset_a[1,n]$.
Since $k \leq \sqrt{n}$, we have $k \in \ADset_a^-(n)$ as desired.

Thus $\J_n$ maps $\ADasetlarge(n)$ into $\ADasetsmall(n)$, and
injectivity is inherited from the behavior of $\J_n$ on $\ADset_1[1,n]$.
\end{proof}

For $a \ge 2$ the set $\ADasetlarge(n)$ is generally significantly smaller than the set $\ADasetsmall(n)$,
see Figure \ref{fig:51} for $a=2$. Since  these sets are contained in $\ADset_1^{+}(n)$
and $\ADset^{-}(n)$, respectively, they are each of size at most $\sqrt{n}$.

\begin{figure}[h]
    \centering
    \includegraphics[scale=0.65]{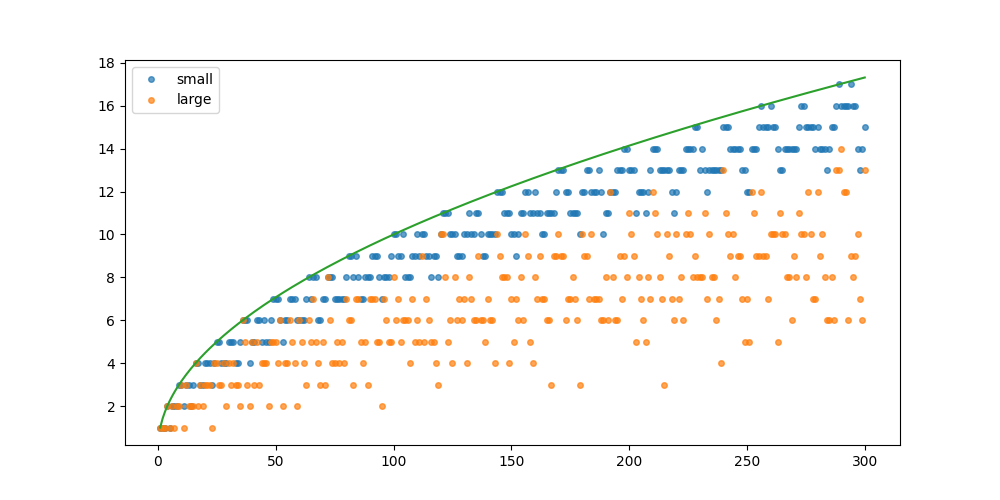}
    \caption{Plot of  the number of small  $2$-floor quotients  $\verts{ \ADset_2^{-}(n) }$ (blue)  versus number of large $2$-floor quotients $\verts{ \ADset_2^{+}(n) }$ (orange) for $n$ up to $300$.
    The upper envelope line is $y = \sqrt{n}$.}
    \label{fig:51}
\end{figure}

%
%
\begin{lem}\label{lem:small-a-floor-quotient-interval}
Let $a \ge 1$.  
For every $n \ge 1$, all  $m \in [1, \sqrt{\frac{n}{a}}]$ belong to  $\ADset_a[1,n]$. 
In particular,
\begin{equation}\label{eqn:minus-lower}
	\verts{ \ADasetsmall(n)} \geq \floor{\sqrt{\frac{n}{a}}}.
\end{equation}
Additionally $\verts{ \ADset_a [1,n]} >  \sqrt{\frac{n}{a}}$ if $n > 1$.
\end{lem}

\begin{proof}
Let $m$ be in the range $1 \le m \le \floor{\sqrt{\frac{n}{a}}}$. 
We wish to show that $m \ADa n$,
so by definition of the $a$-floor quotient relation, it suffices to show that $am \AD an$.

The condition $m \leq \sqrt{n / a}$ implies that $am \leq \sqrt{a n}$.
It follows that $am \AD an$ as desired, by applying Proposition~\ref{prop:interval-complement}

We have $n = m k+r$ where $k = \floorfrac{n}{m} \in \ADset_1^{+}(n)$, and $0 \le r < m$. 
We have $m \le \frac{k}{a}$ since 
\[
k = \floor{\frac{n}{m}}
\geq \floor{ \frac{n}{\floor{\sqrt{{n} / {a}}}}}
\geq \floor{ \frac{n}{\sqrt{{n} / {a}}}} = \floor{\sqrt{an}} \ge a \floor{\sqrt{\frac{n}{a}}} \ge am.
\]
Therefore $r < m = \min ( m, \frac{k}{a})$ so $m \in \ADasetsmall(n)$, as asserted.

For the final assertion, if $n \ge 2$ then $\ADset_a[1,n]$ includes the  element $n \in \ADasetlarge(n)\smallsetminus \ADsetsmall(n)$.
\end{proof}


\begin{lem}\label{lem:53}
For each $a \geq 2$, there is some $n$ such that 
$\ADasetsmall(n)$ is strictly larger in size than $\ADasetlarge(n)$.
\end{lem}
\begin{proof}
Suppose $n = a^2 + a + 1$.
Then $a \in \ADasetsmall(n)$, while $a + 1 = \J_n(a) \not\in \ADasetlarge(n)$.
\end{proof}

%
%
\subsection{Total size of  initial \asafe-floor quotient intervals: heuristic}\label{subsec:52b} 

We  study the total size of the initial $a$-floor quotient interval $\verts{ \ADset_a[1,n] }$,
as $n$ varies. 
In the case $a = 1$ these values are non-decreasing in $n$, and their total 
number is  $2 \sqrt{n} + O(1)$.

We present data for $2$-floor quotients in Figure~\ref{fig:total-count-300}. 
The interval sizes $\verts{ \ADset_2[1,n] }$ are scattered, exhibiting
non-monotonic behavior.
The occurrence of some  scatter of values is not surprising, in that  as $a \to \infty$  the interval $\ADset_{a}[1,n]$ converges to the set  $\sD[1,n]$ consisting of divisors of $n$, 
and the number of divisors of $n$ varies non-monotonically in $n$.

\begin{figure}[h]
\centering
\includegraphics[scale=0.65]{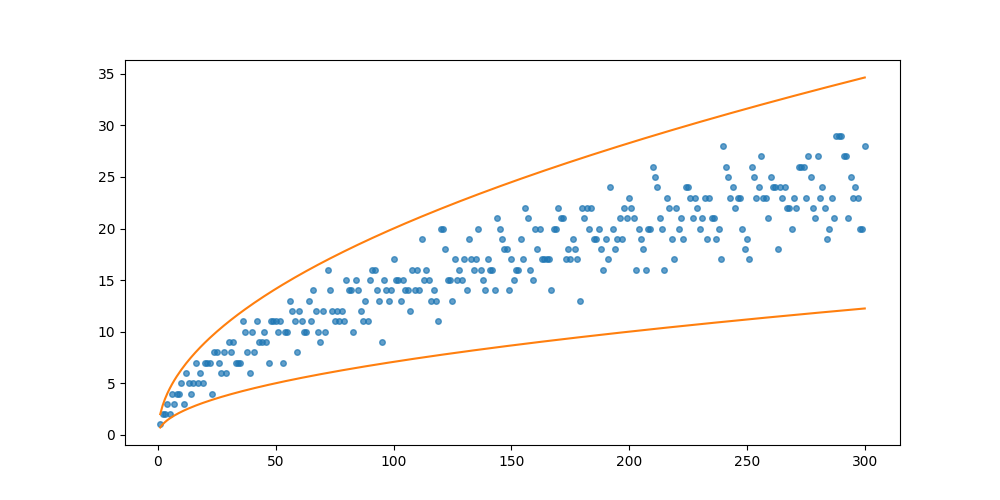}
\caption{Plot of the total number of 2-floor quotients  $\verts{ \ADset_2[1,n] }$ for $n$ up to $300$.
The upper smooth curve is $y = 2 \sqrt{n}$ and the lower curve is  $y = \sqrt{n/2}$.}
\label{fig:total-count-300}
\end{figure}

We formulate a probabilistic heuristic that makes a prediction the (expected) size of the sets $\ADasetlarge(n)$, $\ADasetsmall(n)$ and $\ADset_a[1,n]$,
viewing $a$ as fixed and $n \gg 1$. 

Recall that by the reciprocal-duality property (6) of Theorem~\ref{thm:a-equiv-properties} the value $d$ is an $a$-floor quotient of $n$ if and only if, writing $k = \floorslash{n}{d}$, we have $\floorslash{an}{k} = ad$.
We have three regions to consider.
\begin{enumerate}
\item[(1)] (small region) $1 \le d \le \sqrt{\frac{n}{a}}$.

\smallskip 
In this case we have $d \in \ADaset[1, n]$ by Lemma~\ref{lem:small-a-floor-quotient-interval}. 
Thus $\ADaset[1, n]$ contains exactly
\[
E_{small}= \floor{\sqrt{\frac{n}{a}}} = \sqrt{\frac{n}{a}} + O(1)
\] 
such elements in this region, all of them in $\ADasetsmall(n)$.

\medskip
\item[(2)] (medium region) $\sqrt{\frac{n}{a}} < d \le \sqrt{n}$.

\smallskip
In this region, the cutting length $k = \floorslash{n}{d}$ is in the range $ \sqrt{n} \leq k < \floor{\sqrt{a n}}$. 
The number of integers $k$ in this range is $(\sqrt{a} - 1) \sqrt{n} + O(1)$.

The element $d$ corresponding to $k$ is an $a$-floor quotient if and only if $\floorfrac{an}{k}$ is a multiple of $a$,
in which case  $d = \frac{1}{a} \floorfrac{an}{k}$. 
(Note that over this range of $k$, all values $\floorfrac{an}{k}$ are distinct.) 
As a heuristic, for fixed $n$ we introduce a probabilistic model for accepting  the cutting length $k$ as generating a
``random $a$-floor quotient'' $(d', n)$ as follows.
For each $k$ we replace the value  $r := \floorslash{an}{k} \, (\bmod \, a)$  by an (independently drawn) 
random residue $\tilde{r}\,  (\bmod \, a)$ uniformly distributed in $\{0, 1, \ldots, a - 1\}$. 
If $\tilde{r} \equiv 0 \, (\bmod \, a)$, then we  accept $k$ as giving a ``random $a$-floor quotient,'' 
which would be $d' = \frac{1}{a} (\floor{\frac{an}{k}} - r)$.
The model will treat all  $d'$  as  distinct,  although repeated values of $d'$ could occur. 
That is, for a sample of  the model  we count the number of $k$ in  $ \sqrt{n} \leq k < \floor{\sqrt{a n}}$ having $\tilde{r} \equiv 0 \, (\bmod\, a)$.
The probability that $\tilde{r}$ is a multiple of $a$ is $\frac{1}{a}$,
so the expected number of  ``random $a$-floor quotients'' in this region is
\[
	E_{medium} = \frac{1}{a} \verts{\{k : \sqrt{n} \leq k < \floor{\sqrt{a n}} \}} = \left( \frac{1}{\sqrt{a}} - \frac{1}{a} \right) \sqrt{n} + O\left(\frac{1}{a}\right).
\]

\item[(3)](large region) $\sqrt{n} \le d \le n$.

\smallskip
In this region, the cutting length $k = \floorfrac{n}{d}$ is in the range $1 \leq k \leq \floor{\sqrt{n}}$, which consists of $\sqrt{n} + O(1)$-many integers.
We apply again the  heuristic probabilistic model of ``random $a$-floor quotients''  above for $1 \le k \leq  \floor{\sqrt{n}}$,
finding that  the expected number of  ``random $a$-floor quotients'' assigned from this region is
\[
	E_{large} = \frac{1}{a} \verts{\{k : 1 \leq k \leq \floor{\sqrt{n}} \}} = \frac{1}{a} \sqrt{n} + O\left(\frac{1}{a}\right).
\]
\end{enumerate}

The heuristic (2) (taken together with the rigorous bound (1)) suggests that the number of elements of $\ADasetsmall(n)$ is approximately $(\frac{2}{\sqrt{a}} - \frac{1}{a}) \sqrt{n}$.
The heuristic (3) suggests that the number of elements of $\ADasetlarge(n)$ is approximately $\frac{1}{a} \sqrt{n}$.
Combining the heuristics on the two ranges predicts that the expected number of elements of $\ADset_a[1,n]$ is approximately $2 \sqrt{\frac{n}{a}}$.

In comparison to the heuristic for $\ADset_a[1,n]$,
Lemma~\ref{lem:small-a-floor-quotient-interval} yields a rigorous lower bound 
for the number of elements in $\ADset_a[1,n]$,
that  is half the value predicted combining  (1), (2), (3) 
for the  ``expected'' number of elements. 

\begin{rmk}\label{rmk:54}
(1)  It is an interesting question, not addressed here, to determine information on the size of the scatter of values for
total $a$-floor quotients $\verts{ \ADaset[1,n] }$ visible in Figure \ref{fig:total-count-300}, around their ``expected size'' $f_{a}(n) :=  2 \sqrt{\frac{n}{a}}$. 

(2) One may also ask whether, for  fixed $a$, 
 asymptotic formulas  $\verts{ \ADsetsmall(n) } \sim (\frac{2}{\sqrt{a}} - \frac{1}{a}) \sqrt{n}$ 
and $\verts{ \ADsetlarge(n) } \sim \frac{1}{a} \sqrt{n}$ might hold as $n \to \infty$.

\end{rmk}

%
%
\section{Floor Multiples for \asafe-Floor Quotient Orders}\label{sec:a-floor-multiples}  
\setcounter{equation}{0}

Recall that the $1$-floor multiples of a fixed positive integer $d$ are the integers $n$ such that $d \AD n$.
In  \cite{LagR:23a} it was shown that the  $1$-floor multiples of $d$ form a numerical semigroup.
Here we  extend the results to $a$-floor quotients of $d$.

\begin{defi}\label{def:36}
We call $n$  an {\em $a$-floor multiple of $d$}  if $d$ is an $a$-floor quotient of $n$.
We let  $\ADmult_a(d)$ 
 denote the set of floor multiples of $d$,
 i.e.
\[
\ADmult_a(d) := \{ n\in \NNplus : d \ADa  n \}. 
\]
\end{defi}


 Recall that a  {\em numerical semigroup} is a subset of $\NNplus$
that is closed under addition
and   contains all but finitely many elements of $\NNplus$.
The {\em Frobenius number}  of a numerical semigroup is the largest integer
not belonging to the semigroup.  We determine the Frobenius numbers
and minimal generating sets of these numerical semigroups. 
 For general information on  numerical semigroups and of the
 Frobenius number of a numerical semigroup, see  
 Ram\'{i}rez Alfons\'{i}n~\cite{RamirezA:05} and Assi and Garcia-S\'{a}nchez \cite{AssiG:16}.

%
%
\subsection{Semigroup of 1-floor multiples}\label{sec:41}

The numerical semigroup property was proved for the case $a = 1$ in \cite[Theorem 3.8]{LagR:23a}.

%
%
\begin{thm}[$1$-floor multiple numerical semigroup]
\label{thm:floor-multiple-struct}
The set $\ADmult_1(d)$ of floor multiples of $d$ is a numerical semigroup.
It has the following properties.
\begin{enumerate}
\item 
The largest integer not in the numerical semigroup $\ADmult_1(d)$, its Frobenius number,  is $d^2-1$.
\item
There are exactly $\frac{1}{2} (d-1)(d+2) = \frac{1}{2} (d^2+d - 2) $ positive integers not in $\ADmult_1(d)$.
\item
The  minimal generating set of $\ADmult_1(d)$ has $d$ generators, 
 with $\ADmult_1(d) = \angles{  \gamma_0, \gamma_1, \ldots, \gamma_{d - 1} }$
where
 $\gamma_j = (j + 1)d + j =  j(d + 1) + d$
for $0 \le j \le d - 1$.
\end{enumerate}
\end{thm} 

%
%
\subsection{Semigroup of \asafe-floor multiples}\label{sec:42} 

We treat the case of general $a$. 
%
%
\begin{thm}[$a$-floor multiple numerical semigroup]
\label{thm:a-floor-multiple-struct}
The set $\ADmult_a(d)$ of $a$-floor multiples of $d$
is a numerical semigroup.
It has the following properties.
\begin{enumerate}
\item [(1)]
The largest integer not in $\ADmult_a(d)$, its Frobenius number,  is 
$(d-1)(ad+1)$.

\item[(2)]
The number of positive integers not in $\ADmult_a(d)$ is $\frac12(d-1)(ad+2)$.

\item [(3)]
The minimal generating set of $\ADmult_a(d)$ is
$
\angles{ \gamma_0, \gamma_1, \ldots, \gamma_{d - 1} }
$
where 
\[
	\gamma_j = (ja + 1) d  + j = j(ad+1) + d \quad \mbox{for} \quad 0 \le j \le d - 1.
\]
\end{enumerate}
\end{thm}

The proof of Theorem \ref{thm:a-floor-multiple-struct}
follows the same outline as the $a=1$ case given in \cite[Theorem 4.2]{LagR:23a}.

\begin{lem}\label{lem:a-multiples-add} 
The set $\ADmult_a(d)$ of $a$-floor multiples of $d$ is closed under addition.
\end{lem}
\begin{proof}
By definition of the relation $\fquo{a}$, 
the map $n\mapsto an$
gives a bijection from $\ADmult_a(d)$ to the intersection
$a\NNplus \cap \ADmult_1(ad)$.
This map respects addition, so it suffices to check that the image of $\ADmult_a(d)$ is closed under addition.
The intersection $a\NNplus \cap \ADmult_1(ad)$
is closed under addition
because $a\NNplus$ and $\ADmult_1(ad)$ are both closed under addition.
\end{proof}

\begin{lem}\label{lem:a-multiples-elem} 
Fix positive integers $a$ and $d$.
Let $\ADmult_a(d)$ denote the set of $a$-floor multiples of $d$.
\begin{enumerate}
\item 
For any positive integer $n$,
write $n = kd + j$ for  unique integers $j, k$ satisfying ${0 \leq j < d}$.
Then $n \in \ADmult_a(d)$ if and only if $aj < k$.

\item
If $n\geq ad^2 - (a-1) d$,
then $n \in \ADmult_a(d)$.

\end{enumerate}
\end{lem}

\begin{proof}
(1) 
The $a$-reciprocal duality property, Theorem~\ref{thm:a-equiv-properties} (6), states that
\[
	d \ADa n  \qquad\text{if and only if}\qquad
	\frac{n}{\floor{n/d}} < d + \frac1{a}.
\] 
Given that $n = kd +j$ with $j \in \{0, 1, \ldots d - 1\}$,
we have
\[
	\floor{\frac{n}{d}} = k = \frac{n - j}{d}.
\]
The expression $n / \floor{n/d}$ can be rewritten as
\[
	\frac{n}{\floor{n / d}} = \frac{n}{(n - j) / d} = d \left( 1 + \frac{j}{n - j} \right) = d\left(1 + \frac{j}{kd}\right).
\]
Thus the $a$-reciprocal duality property implies that
\[
	d \ADa n \qquad\text{if and only if}\qquad
	d\left(1 + \frac{j}{kd}\right) < d\left(1 + \frac1{ad}\right),
\]
which is equivalent to the condition $aj < k$ as claimed.
 
(2) 
Suppose  $n \ge ad^2 - (a-1)d$,
so 
\[
	n = kd + j \qquad\text{for } j \in \{0,1,\ldots, d-1\}
	\text{ and } k \geq ad - a + 1.
\] 
We then have
\[
	aj \leq a(d - 1) < k.
\]
Thus $d \ADa n$ by part (1).
\end{proof}


\begin{proof}[Proof of Theorem \ref{thm:a-floor-multiple-struct}]
Lemma \ref{lem:a-multiples-add} shows $\ADmult_a(d)$ is an additive semigroup, 
and  Lemma \ref{lem:a-multiples-elem} (2)  then shows $\ADmult_a(d)$ is a  numerical semigroup.

(1) The condition in Lemma~\ref{lem:a-multiples-elem} (1) implies that $n=ad^2 -(a-1)d -1$
is not in $\ADmult_a(d)$; in this case $n = kd+j$ has
$k= a(d-1)$ and  $j=d-1$ with $aj =  k$. 
On the other hand, Lemma~\ref{lem:a-multiples-elem} (2) shows that   all integers $n \ge ad^2-(a-1)d$
belong to $\ADmult_a(d)$.

(2) To count the size of $\NNplus \setminus \ADmult_a(d)$,
we apply the condition in Lemma~\ref{lem:a-multiples-elem} (1).
That condition  counts  the set  $n = kd + j$ with $0 \le k \le ad$ and $0 \le j \le d-1$
such that $aj \ge k \ge 0$, excluding the pair $(j,k)=(0,0)$, because $n=0$ is excluded from $\NNplus$.
  We sum over $0 \le j \le d-1$, noting for  $j=0$ there
are no solutions, and there are $aj+1$ solutions for $1 \le j \le d-1$.  We obtain
\[
	\verts{\NNplus \setminus \ADmult_a(d)} = \sum_{j=1}^{d-1} (aj+1) = a\left( \frac{d(d-1)}{2}\right) + d-1 = \frac{1}{2} ( d-1)(ad+2). 
\]

(3) Recall that $\gamma_i = ((i-1)a + 1)d + i-1$.
Each generator $\gamma_i $ has $\gamma_i= kd + j$  with $j = i - 1$ and $k = (i-1)a+1$, so $aj < k$, whence   
by Lemma~\ref{lem:a-multiples-elem} (1)  $\gamma_i $  is in $\ADmult_a(d)$.
Thus we have the containment 
$\angles{ \gamma_1,\ldots, \gamma_d } \subset \ADmult_a(d)$.
The reverse containment $\ADmult_a(d) \subset \angles{\gamma_1,\ldots, \gamma_d }$
 holds because 
$d \in \ADmult_a(d)$ and 
the generating set
$\gamma_1,\gamma_2,\ldots,\gamma_d$ contains the minimal element of $\ADmult_a(d)$ in each residue class modulo $d$,
which also follows from Lemma~\ref{lem:a-multiples-elem}.

It remains to prove  that $\gamma_1,\ldots,\gamma_d$ is  minimal as a generating set  
of $\ADmult_a(d)$,
i.e. that no strict subset of $\{\gamma_1,\ldots, \gamma_d\}$ generates $\ADmult_a(d)$. 
Since $\gamma_1 < \gamma_2 < \cdots < \gamma_d$,
it suffices to show that $\gamma_i$ is not contained in the numerical semigroup $\angles{ \gamma_1, \ldots, \gamma_{i-1} }$.

Suppose for the sake of contradiction that $\gamma_i \in \angles{ \gamma_1, \ldots, \gamma_{i-1} }$.
Then for some $\ell < i$, 
we have 
$
\gamma_i - \gamma_\ell  
\in \angles{\gamma_1, \ldots, \gamma_{i-1} }.
$
However, 
we have the inclusion
$\angles{\gamma_1, \ldots, \gamma_{i-1} }\subset \ADmult_a(d)$,
and 
Lemma~\ref{lem:a-multiples-elem}(2)  implies that 
\[
\gamma_i - \gamma_\ell  
= (i- \ell)ad - (i-\ell) = kd+j
 \not\in \ADmult_a(d) ,
\]
since $aj=k$.
This contradiction implies that $\gamma_i \not \in \angles{\gamma_1, \ldots, \gamma_{i-1} }$,
so $\gamma_1,\ldots,\gamma_d$ is the minimal  generating set  
of $\ADmult_a(d)$ as claimed.
\end{proof} 

%
%
\section{Average size of intervals \texorpdfstring{$\verts{ \ADset_a[1,n] }$}{|Qa[1,n]|}}\label{sec:7} 

We use the structure of floor multiples to estimate the average size of $a$-floor quotient initial intervals for fixed $a$ and variable $n$.
For context, recall that when $a = 1$ the number of $1$-floor quotients satisfies
\[
	\verts{ \ADset_1[1, n] } = 2 \sqrt{n} + O(1) .
\]
Hence, the average size of $\verts{ \ADset_1[1, n] }$ in the range $1 \leq n \leq x$ is asymptotic to
\[
	\frac{1}{x} \sum_{n = 1}^x \verts{ \ADset_1[1,n] } 
	= \frac{1}{x} \left( \int_{t = 0}^x 2 \sqrt{t} \, dt +  O(x)\right)
	= \frac{4}{3} \sqrt{{x}} + O(1).
\]
  
%
%
\subsection{Averaged number of \asafe-floor quotients}
\label{subsec:70} 

We have the following estimate, averaging  over all  $n \le x$.

\begin{thm}\label{thm:52}
For all integers $a \ge 1$ and all $x \ge 1$,
\begin{equation} 
\label{eqn:averaged}
\frac{1}{x} \sum_{n=1}^x \verts{ \ADset_a[1,n] } = \frac{4}{3} \sqrt{\frac{x}{a}} + O \left(  \log x +a \right),
\end{equation}
in which the $O$-constant is independent of $a$ and $x$.
\end{thm}

The proof of Theorem \ref{thm:52}  is   based on study of an auxiliary counting function.

\begin{defi}\label{def:a_mult_count}
The {\em $a$-floor-multiple counting function for $d$}, denoted  $\multcount_{a,d}(x)$, is
\begin{equation}
\multcount_{a,d}(x) 
:= \verts{ \{ n: \, d \ADa n, \, n \le x\} }
= \verts{ \ADmult_a(d) \cap \{ n \leq x\} }.
\end{equation}
\end{defi} 

The usefulness of this function is based on  the 
identity
\begin{equation}\label{eqn:key-identity}
 \sum_{n=1}^x \verts{ \ADset_a[1,n] } = \sum_{d=1}^x \multcount_{a,d}(x),
\end{equation}
 which is valid for all $a \ge 1$  and $x \ge 1$.

The result \eqref{eqn:averaged} is proved in three lemmas in Section \ref{subsec:73} which
estimate the contributions of $\multcount_{a,d}(x)$ on the right side of \eqref{eqn:key-identity} for  three regions of $d$ that parallel the
regions of the heuristic in Section \ref{subsec:52b}. 

The estimate in Theorem \ref{thm:52} is compatible with the heuristic for $\verts{ \ADset_a[1,n]}$ being of size $2 \sqrt{\frac{n}{a}}$, 
as given in Section  \ref{subsec:52b}, 
in the sense that the integrated version of the heuristic 
matches the main term on the  right side of  \eqref{eqn:averaged}. 
That is, 
\begin{equation}
\frac{1}{x} \sum_{n=1}^x  2 \sqrt{\frac{n}{a}} = \frac{4}{3} \sqrt{\frac{x}{a}} + O \left( \frac{1}{\sqrt{ax}} \right).
\end{equation}

%
%
\subsection{Counting \asafe-floor multiples}
\label{subsec:72} 

We give a formula for the $a$-floor multiple counting function values $\multcount_{a,d}(x).$
%
%
\begin{prop}\label{prop:pi-floor-recursion}
Let $a\ge 1$ and $d\ge 1$. 
Then, for all $x \ge 1$, 
\begin{equation}\label{eqn:pi-recursion}
\multcount_{a,d}(x) = \sum_{j=0}^{d-1} \floor{\frac{x- j(ad+1)}{d}}^{+},
\end{equation}
where $\floor{x}^{+} = \max \{ \floor{x}, 0\}$. 
\end{prop}

\begin{proof}
The structure of $a$-floor multiples $n$ given in Theorem~\ref{thm:a-floor-multiple-struct} shows the minimal $a$-floor multiple
in each residue class $j \, (\bmod \, d)$ (taking $0 \le j \le d-1$) is
$\gamma_j = j(ad + 1) + d$. 
After this value, every shift by a multiple of $d$ gives an $a$-floor quotient. 
The $a$-floor multiples of $d$ 
in this residue class, no larger than $x$, are counted  exactly by
$\floorfrac{x - j(ad + 1)}{d}^{+}$. 
Summing over all residue classes $(\bmod \, d)$ gives the result.
\end{proof}

The next result gives a  size estimate for $\multcount_{a,d}(x)$, 
which exhibits a phase transition at $x = ad^2$.  
For fixed $(a, d)$ it initially increases in $x$ following a parabola, and then flattens out to linear growth in $x$, when the numerical semigroup takes all values. 
\begin{prop}
\label{prop:sigma-estimate}
For all $a \ge 1$ and $d \ge 1$, there holds
\[
	\multcount_{a,d}(x) = \begin{cases}
	\frac{1}{2}a \left(\frac{x}{ad}\right)^2 + O \left(ad \right) &\quad \text{if} \quad 0 < x \le ad^2, \\[0.5em]
    x  - \frac{1}{2} ad^2 + \frac{1}{2} ad + O(d) &\quad \text{if}\quad ad^2 \leq x.
	\end{cases}
\]
The $O$-constants are independent of $a$ and $d$.
\end{prop}

\begin{proof}
For  the saturation region above the phase change, $x \geq ad^2$, we have
\[
x \ge ad^2 \ge  (d-1)(ad+1)= ad^2 -(a-1)d -1.
\]
Consequently $ \floorfrac{x- j(ad+1)}{d}^{+}= \floorfrac{x- j(ad+1)}{d}$ holds  for 
the full range $0 \le j \le d - 1$. 
Then Proposition \ref{prop:pi-floor-recursion} gives  
\begin{align*}
\multcount_{a,d}(x) 	 
	&=  \sum_{j=0}^{d-1} \floor{\frac{x - j(ad+1)}{d}} \\
	&= \sum_{j=0}^{d-1} \left( \frac{x}{d} - \frac{j}{d}(ad + 1) + O(1)\right)\\
	 &= x - \frac{1}{2}(d - 1) \left(ad + 1\right) + O \left( d\right)  \\
	 &= x- \frac{1}{2} a d^2 + \frac{1}{2} ad + O \left( d \right).
\end{align*}

For the region below the phase change, where $\ell ad \le x < (\ell +1) ad$  for some $\ell$ with  $0 \le \ell < d$, 
we have $\floorfrac{x- j(ad+1)}{d}^{+} = \floorfrac{x- j(ad+1)}{d} $ for $0 \le j \le \ell-1$, 
but equals $O\left( a \right) $ for $j = \ell$, and equals $0$ for $j \ge \ell + 1$.
 Consequently 
\begin{align*}
\multcount_{a,d}(x) 	 &= \sum_{j=0}^{d-1} \floor{\frac{x- j(ad+1)}{d}}^{+}\\
	&= \left(  \sum_{j=0}^{\ell-1} \floor{\frac{x - j(ad+1)}{d}} \right)  + O \left( a  \right)\\
	& = \sum_{j=0}^{\ell-1} \left( \floorfrac{x}{d} - \frac{j}{d}(ad + 1)  + O(1)\right) + O \left( a \right).
\end{align*}
Since $\ell  \floorfrac{x}{d}= a\ell^2 +O\left(a\ell \right)  = a(\frac{x}{ad})^2 +O\left(a\ell \right)$, we obtain
\begin{align*}
\multcount_{a,d}(x) 	 
	 &= \ell \floorfrac{x}{d} - \frac{(\ell-1)\ell}{2} \left( a+ \frac{1}{d}\right) +O \left( a +\ell \right) \\
	  &=  a(\frac{x}{ad})^2 - \frac{1}{2} a \, \left(\frac{x}{ad}\right)^2+ O \left( ad \right),\\
	  &=\frac{1}{2} a ( \frac{x}{ad})^2 +O\left( ad \right), 
\end{align*}
using $\verts{ \frac{x}{ad}- \ell } \le 1$ at the second to last line.

The sum for $\sigma_{a,d}(x)$  is vacuous for $x \le 0$.
\end{proof} 

%
%
\subsection{Total size of initial \asafe-floor quotient intervals}
\label{subsec:73} 

We estimate the size of sums $\multcount_{a,d}(x)$ over the small, medium, and large regions of $d$, where $1 \le d \le x$. 
Some of these sums may be empty for small $x$.

\begin{lem}[Small region]
\label{lem:small-region}
 For any integer $a \ge 1$ and any real $x \ge 1$, 
\begin{equation}\label{eqn:small-region}
\sum_{d=1}^{ \sqrt{{x} / {a}}} \multcount_{a,d}(x) = \frac{5}{6\sqrt{a}} x^{3/2} + O \left(x \right),
\end{equation}
where the constant implied by the O-symbol   is independent of both $a$ and $x$.
\end{lem}

\begin{proof}
For $1 \le d \le \sqrt{\frac{x}{a}}$ we have $x \ge a d^2 $.
Applying Proposition~\ref{prop:sigma-estimate} yields
\begin{align*}
\sum_{d=1}^{ \sqrt{x/a} } \multcount_{a,d}(x) 
&=  \sum_{d=1}^{ \sqrt{x/a} } \left( x - \frac{1}{2} a d^2 + \frac{1}{2}ad + O(d)  \right)\\
&= \left(\frac{x^{3/2}}{\sqrt{a}} +O\left( x \right) \right) - \left(\frac{1}{6}  \frac{x^{3/2}}{\sqrt{a}} +O\left(x \right) \right)+
\left(\frac{1}{4}x + O\left(  \sqrt{ax}  \right)\right) + O\left( \frac{x}{a}\right) \\
&= \frac{5}{6\sqrt{a}} x^{3/2} + O \left(x  \right),
\end{align*}
which is \eqref{eqn:small-region}.
\end{proof}

%
%
\begin{lem}[Medium region]\label{lem:middle-region}
For any integer $a \ge 1$ and any real $x \ge 1$, 
\begin{equation}\label{eqn:middle-region}
\sum_{d >   \sqrt{{x} / {a}}}^{ \sqrt{x}}\multcount_{a,d}(x)  =
\frac{1}{2\sqrt{a}} x^{3/2} - \frac{1}{2a}  x^{3/2}+ O \left( x \sqrt{a} + a\sqrt{x} \right),
\end{equation}
where the constant implied by the $O$-symbol is independent of both $a$ and $x$.
\end{lem}

\begin{proof} 
By hypothesis $\sqrt{{x}/ {a}} < d  \le \sqrt{x}$, so that $x < ad^2 \le ax$, whence
\begin{equation}\label{eqn:short}
 \floorfrac{x}{ad}\le \frac{x}{ad} < d \le \frac{x}{d}.
\end{equation}
Recall that for all  integers $a \ge 1$ and any $x \ge 1$, 
\begin{align*}
\sum_{d> \sqrt{{x} / {a}}}^{\sqrt{x}}  \multcount_{a,d}(x) &=
\sum_{d > \sqrt{{x} / {a}}}^{\sqrt{x}} \left( \sum_{j=0}^{d-1}  \floorfrac{x-j(ad+1)}{d}^{+} \right) \\
&= \sum_{d > \sqrt{{x} / {a}}}^{\sqrt{x}} \left( (\sum_{j=0}^{\floorslash{x}{ad}}  \floorfrac{x-j(ad+1)}{d}  )  
+ O\left(\frac{1}{d} \floorfrac{x}{ad}\right)\right).
\end{align*}
where the cutoff in the inner sum can be made since $x - j(ad + 1) < 0$ holds for all $j \ge \floorfrac{x}{ad} + 1$.  
Only the  term $j=\floorfrac{x}{ad}$ can contribute to the remainder term, 
and only when 
\[
0 > x - j(ad + 1) = (x- \floorfrac{x}{ad} ad) - \floorfrac{x}{ad} \ge - \floorfrac{x}{ad}
\]
requiring a correction term at most $O\left( \frac{1}{d} \floorfrac{x}{ad}\right)$.
The implied $O$-constant is independent of $x, a$ and $d$.
Consequently,
\begin{align*}
\sum_{d>\sqrt{{x} / {a}}}^{\sqrt{x}}  \multcount_{a,d}(x) 
&= \sum_{d > \sqrt{{x} / {a}}}^{\sqrt{x}} \left( \sum_{j=0}^{\floorslash{x}{ad}}\left(  \floor{\frac{x}{d}} - j(a+ \frac{1}{d})  + O\left( 1\right) + O \left( \frac{1}{d} \floorfrac{x}{ad} \right) \right)\right)\\
 &=\sum_{d > \sqrt{{x} / {a}}}^{\sqrt{x}} \left( \sum_{j=0}^{\floorslash{x}{ad}}\left(  \floor{\frac{x}{d}} - j(a+ \frac{1}{d})  \right) 
 + O\left( 1\right) \right),
\end{align*}
using $0 < \floorfrac{x}{ad} < d$ from \eqref{eqn:short} in the last line. 
Next, using this bound again on the  first line, 
\begin{align*} 
\sum_{d > \sqrt{{x} / {a}}}^{\sqrt{x}}  \multcount_{a,d}(x) 
 &= \sum_{d > \sqrt{{x} / {a}}}^{\sqrt{x}} \left( \sum_{j=0}^{\floorslash{x}{ad}}\left(  \floor{\frac{x}{d}} - j(a+ \frac{1}{d})  \right) \right) +
 O\left( \sum_{d > \sqrt{{x} / {a}}}^{\sqrt{x}} d  \right)\\
 &=  \sum_{d > \sqrt{{x} / {a}}}^{\sqrt{x}} \left( \sum_{j=0}^{\floorslash{x}{ad}}\left(  \floor{\frac{x}{d}} - j(a+ \frac{1}{d}) \right) \right)+ O\left(x( 1- \frac{1}{a}) + \frac{x}{\sqrt{a}} \right).
 \end{align*}

We evaluate the inner sums to obtain
\begin{align*}
\sum_{d>\sqrt{{x} / {a}}}^{\sqrt{x}}  \multcount_{a,d}(x) 
 &= \sum_{d > \sqrt{{x} / {a}}}^{\sqrt{x}} \left( (\floorfrac{x}{ad} + 1) \floorfrac{x}{d} - (a + \frac{1}{d} ) \frac{1}{2} \floorfrac{x}{ad} ( \floorfrac{x}{ad} +1) \right)
 +O \left(x (1 - \frac{1}{a}) + \frac{x}{\sqrt{a}}  \right) \\
&= \sum_{d > \sqrt{{x} / {a}}}^{\sqrt{x}} \left((\floorfrac{x}{ad} \floorfrac{x}{d} - \frac{a}{2} \floorfrac{x}{ad}^2) + O\left( a (1+\floorfrac{x}{ad} )\right) \right) + 
O \left(x(1 - \frac{1}{a}) + \frac{x}{\sqrt{a}}) \right).
\end{align*}

We remove the floor functions, adding  new  remainder terms, and simplify,  obtaining
\begin{align*}
\sum_{d > \sqrt{{x} / {a}}}^{\sqrt{x}}  \multcount_{a,d}(x)  
&= \frac{1}{2a} \sum_{d > \sqrt{{x} / {a}}}^{\sqrt{x}}\frac{x^2}{d^2} +
 \sum_{d > \sqrt{{x} / {a}}}^{\sqrt{x}} O \left( a+\floorfrac{x}{d}\right)  +O \left(  x( 1-\frac{1}{\sqrt{a}})       \right) \\
&= \frac{x^2}{2a} \sum_{d > \sqrt{{x} / {a}}}^{\sqrt{x}} \frac{1}{d^2} + O\left(\sqrt{x} a + x \log(\sqrt{a}) + x( 1-\frac{1}{\sqrt{a}}) \right)\\
&= \frac{x^2}{2a}\left( \sqrt{\frac{a}{x}} - \frac{1}{\sqrt{x}} \right) + O \left(x (1+ \log a) +  a\sqrt{x}  \right)\\
&= \frac{x^2}{2a}\left( \sqrt{\frac{a}{x}} - \frac{1}{\sqrt{x}} \right) + O \left(x \sqrt{a} + a \sqrt{x}  \right),
\end{align*}
which gives \eqref{eqn:middle-region} on the domain $x \ge 1$ and integer $a \ge 1$.
\end{proof} 

\begin{lem}[Large region]\label{lem:large-region}
For any $a \ge 1$ and any  real $x \ge 1$, 
\begin{equation}\label{eqn:large-region}
\sum_{d >   \sqrt{x}    }^{x}\multcount_{a,d}(x)  =
 \frac{1}{2a}  x^{3/2}+ O (x \log x +ax),
\end{equation}
where the constant implied by the $O$-symbol is independent of both $a$ and $x$.
\end{lem}

\begin{proof}
The arguments parallel the last lemma, except the estimates of remainder terms change,
since \eqref{eqn:short} does not hold; 
instead $\sqrt{x} < d  \le x$. 
Note that $ad > x$ can occur.

As before we obtain
\begin{align*}
\sum_{d> \sqrt{x}}^{x}  \multcount_{a,d}(x)
&= \sum_{d > \sqrt{x}}^{x} \left( \sum_{j=0}^{\floorfrac{x}{ad}}  \floor{\frac{x-j(ad+1)}{d}}  +O\left(\frac{1}{d} \floorfrac{x}{ad}\right)\right)\\
&=\sum_{d > \sqrt{x}}^{x} \left( \sum_{j=0}^{\floorfrac{x}{ad}}  \floorfrac{x}{d} - j(a + \frac{1}{d}) +O(1)  +O \left(\frac{1}{d}\floorfrac{x}{ad}\right)\right),
\end{align*}
But now we have
\begin{align*} 
\sum_{d> \sqrt{x}}^{x}  \multcount_{a,d}(x)
&= \sum_{d > \sqrt{x}}^{x} \left( \sum_{j=0}^{\floorfrac{x}{ad}}\left(  \floorfrac{x}{d} - j(a+ \frac{1}{d})  \right) \right)
 +O\left( \sum_{d > \sqrt{x}}^{x}( \floorfrac{x}{ad} +  \frac{1}{d} \floorfrac{x}{ad}^2) \right)\\
&=\sum_{d > \sqrt{x}}^{x} \left( \sum_{j=0}^{\floorfrac{x}{ad}}\left(  \floorfrac{x}{d} - j(a+ \frac{1}{d})  \right) \right)+
 O\left( \frac{x}{a} \log x + \frac{x}{2a^2}\right), 
\end{align*}
in which the implied $O$-constants are independent of $x$ and $a$. 
The remainder term above is estimated using
\[
\sum_{d> \sqrt{x}}^{x} \frac{x^2}{a^2 d^3} = \frac{x^2}{a^2} \int_{\sqrt{x}}^x \frac{1}{t^3} dt + O(\frac{\sqrt{x}}{a^2} )
= \frac{x}{2a^2} + O(\frac{\sqrt{x}}{a^2}).
\]

Evaluating the inner sums yields 
  \begin{align*}
\sum_{d>\sqrt{x}}^{x}  \multcount_{a,d}(x) 
 &= \sum_{d > \sqrt{x}}^{x} 
 \left( (\floorfrac{x}{ad} \floorfrac{x}{d} - \frac{a}{2} \floorfrac{x}{ad}^2) + O\left( a (1 + \floorfrac{x}{ad} )\right) \right) + O \left( \frac{x}{a} (\log x +1) \right).
\end{align*}
We remove the floor functions, adding  new remainder terms, and simplify, obtaining
\begin{align*}
\sum_{d > \sqrt{x}}^{x}  \multcount_{a,d}(x)  &= \frac{1}{2a} \sum_{d > \sqrt{x}}^{x}\frac{x^2}{d^2} + 
\sum_{d > \sqrt{x}}^{x} O \left(a+ \floorfrac{x}{d}\right) +O \left( \frac{x}{a} \log x \right)\\
&= \frac{x^2}{2a} \sum_{d > \sqrt{x}}^{x}\frac{1}{d^2} + O\left(x\log x  +ax\right)\\
&= \frac{x^2}{2a} \left( \frac{1}{\sqrt{x}} - \frac{1}{x}+O(\frac{1}{x^{3/2}} )\right)+ O \left(x\log x + ax\right), 
\end{align*}
which implies \eqref{eqn:large-region}, with the constant  implied by the $O$-symbol independent of both  $x$ and $a$,
over  the parameter domain $D= \{ (x, a): \, 1 \le x < \infty,\,\, 1\le a < \infty, \, a \in \NN^{+} \}$. 
\end{proof}

%
%
\subsection{Proof of Theorem~\ref{thm:52}}\label{subsec:74} 

\begin{proof}[Proof of Theorem \ref{thm:52}]
The  result 
\begin{equation}
    \sum_{n=1}^x \verts{ \ADset_a[1,n] } = \frac{4}{3 \sqrt{a}} x^{3/2} + O \left( x\log x + ax\right).
\end{equation}
follows from the identity \eqref{eqn:key-identity}, combining the
estimates of the three Lemmas \ref{lem:small-region}, \ref{lem:middle-region} and \ref{lem:large-region}.
The constant implied by the $O$-symbol is independent of both  $x$ and $a$,
over the parameter domain $D = \{ (x, a): \, 1 \le x < \infty,\,\, 1\le a < \infty, \, a \in \NNplus \}$. 
Since $x \ge 1$, dividing both sides by $x$ yields the result.
\end{proof}

\begin{rmk}\label{rem:78}
We have the identity, valid for $1 \le y \le x$, 
\begin{equation}
\sum_{d=1}^y \sigma_{a,d}(x)= \sum_{n=1}^x \verts{ \ADset_a[1, n] \cap \{1, 2, \ldots, y\} }.
\end{equation}
We may then rewrite  Lemma \ref{lem:small-region}
as asserting, for  any $a \ge 1$ and $x \ge 2$, 
\begin{equation}\label{eqn:small-region2}
\sum_{n=1}^x \verts{ \ADset_a[1, n] \cap [1, \sqrt{\frac{x}{a}}] } = \frac{5}{6\sqrt{a}} x^{3/2} + O \left(x \right).
\end{equation}

On the other hand, one can prove
\begin{equation}\label{eqn:small-region2r}
\sum_{n=1}^x \verts{ \ADset_a[1, n] \cap [1, \sqrt{\frac{n}{a}}] } = \frac{2}{3\sqrt{a}} x^{3/2} + O \left(x \right),
\end{equation}
using the fact that $\verts{ \ADset_a[1, n] \cap [1, \sqrt{\frac{n}{a}}] } = \sqrt{\frac{n}{a}} + O(1)$
holds, using  the (rigorous) small region bound (1) given in Section \ref{subsec:52b}. 
Combining \eqref{eqn:small-region2r}  with   Theorem \ref{thm:52}, one 
then deduces
\begin{equation}\label{eqn:small-region3}
\sum_{n=1}^x \verts{ \ADset_a[1, n] \cap [\sqrt{\frac{n}{a}}], n] } = \frac{2}{3\sqrt{a}} x^{3/2} + O \left(x\log x + ax \right).
\end{equation}
It follows  that the median of the distribution of elements of $Q_a[1, n]$, viewed as a subset of $\sA[1,n]$,  is, on average  close to $\sqrt{\frac{n}{a}}$.
\end{rmk}

%
%
\section{Concluding Remarks}\label{sec:concluding} 
\setcounter{equation}{0}

We conclude with  three directions for further investigation.

\subsection{\asafe-floor quotient \Mobius{} functions}\label{subsec:81}

The zeta function and \Mobius{} functions of a partial order were introduced in Rota~\cite{Rota:64}, though they were long studied for distributive lattices, see 
Weisner~\cite{Weisner:35} and Ore~\cite[Chaps. 10, 11]{Ore:62}.

The two-variable \Mobius{} function of the $1$-floor quotient partial order was studied in \cite{LagR:23a}.
The results there showed the function $\muAD(1, n)$ has infinitely many sign changes. 
Numerical studies indicate it takes large values,
suggesting that $\verts{ \muAD(1,n) } \ge n^{0.6}$,
but it has not even been proved that $\verts{ \muAD(1,n) }$
is unbounded.
It is an interesting problem  to study the  \Mobius{} functions of the
 $a$-floor quotient partial orders $\ADset_a = (\NNplus, \ADa)$,
 particularly addressing the \Mobius{} functions of 
 their initial intervals $ \ADset_a[1,n]$. 
 Here we note the following facts.
 
 \begin{enumerate}
 \item
Theorem~\ref{thm:interval-hierarchy-10} implies that 
as $a$ increases,
the value of $\muADa(d,n)$  starts at  $\muAD(d,n)$ and arrives at the classical \Mobius{} function value $\mu(\frac{n}{d})$
for  $a \ge \floor{ \frac{n}{d}}$, defining $\mu(x)=0$ if $x$ is not an integer.
 \item
 One can show a polynomial upper bound on the size of the $a$-floor quotient
 \Mobius{} function $\verts{ \muADa(d, n) } \le (\frac{n}{d})^{\alpha_0}$ 
 where $\alpha_0 \approx 1.729$ is the unique solution  to  $\zeta(\sigma)=2$ on the region $\sigma>1$, using methods similar to that 
 used in \cite{LagR:23a} for the $1$-floor quotient. The upper bounds for
 the $1$-floor quotient $\muAD(1,n)$ used information on the number of chains connecting
 $1$ to $n$. Since the $1$-floor quotient order refines the $a$-floor quotient order,
 each $a$-floor quotient interval has fewer chains than that of the $1$-floor quotient
 so the $1$-floor quotient upper bounds apply.
\item
The behavior of  the \Mobius{} function $\muADa(d,n)$ of the $a$-floor quotient order may exhibit new features over that of the $1$-floor quotient.
As observed in Section \ref{sec:a-duality},  the duality involution for the $1$-floor quotient breaks down for the $a$-floor quotient for $a \ge 2$.
\end{enumerate}

Examples show  the convergence of $\muADa(d,n)$ to $\mu(d,n)$, as $a \to\infty$, is not monotonic on $a$.
In fact there exist  $d , n$ and $a > 1$ such that
\[
	|\mu_1(d,n)|  < |\mu_a(d, n)| .
\]
\begin{exa}\label{exa:81}
\label{exa:mu-13}
For  $(d, n) = (1, 13)$ we have $\muAD(1,13) = 0$ and $\muADof{2}(1, 13) = 1$, while $\mufquo{a}(1,13) = \mu(13)= -1$ holds  for all sufficiently large $a$ 
(in particular, for $a \geq \floorfrac{13}{2} = 6$ by Theorem \ref{thm:interval-hierarchy-10}).
A table listing all values of $\muADa(1,13)$, as $a$ varies, is given below.
\[
 \begin{array}{c|cccccccc}
 a & 1 & 2 & 3 & 4 & 5 & 6 & 7 & \cdots \\
 \hline
 \muADa(1,13) & 0 & 1 & 1 & 0 & 0 & -1 & -1 & \cdots
 \end{array}
\]
\end{exa}
\begin{exa}\label{exa:82}
\label{exa:mu-21}
For $(d, n) = (1, 21)$, we have $\muAD(1, 21) = -1$ and $\muADof{2}(1, 21) = 2$,
while $\muADa(1, 21) = \mu(21) = 1$ for sufficiently large $a$.
The table below lists all values of $\muADa(1, 21)$, as $a$ varies.
\[
 \begin{array}{c|cccccccccc}
 a & 1 & 2 & 3 & 4 & 5 & \cdots & 9 & 10 & 11 & \cdots \\
 \hline
 \muADa(1, 21) & -1 & 2 & 3 & 2 & 2 & \cdots & 2 & 1 & 1 & \cdots
 \end{array}
\]
\end{exa}


\subsection{Incidence algebras of floor quotient orders}
\label{subsec:82}

The incidence algebra of a poset ${(\NNplus, \preccurlyeq)}$ is the set of $\ZZ$-valued functions on the set of pairs 
$(d,n) \text{ such that } d \preccurlyeq_1 n$
under the operations of addition and convolution,
see Smith~\cites{Smith:67,Smith:69}. 
For general structures of incidence algebras over a ring $R$, 
which are in general non-commutative associative algebras, 
see Stanley~\cite{Stanley:70} and Spiegel and O'Donnell~\cite{SpiegelO:97}. 
In general, zeta  functions and \Mobius{} functions of (locally finite) partial orders are members of the incidence algebra.

Recall that the $1$-floor quotient incidence algebra  $\sI(\ADset_1[1,n])$ can be realized as a noncommutative ring of square integer matrices of size $|\ADset_1[1,n]|$.
It has been shown by several authors that $s_1(n):=|\ADset_1[1,n]|$ is of size $s_1(n)=2\sqrt{n} + O(1)$,
(\cite{Cardinal:10}, \cite{Heyman:19}, \cite[Corollary 4.4]{LagR:23a}).
Allowing $\QQ$-coefficients, it was shown in \cite[Theorem 4.15]{LagR:23a}
that the incidence algebra of $\ADset_1[1, n]$
has  dimension $\dim_{\QQ}( \sI_n (\QQ)) \sim \frac{16}{3} n^{3/4}$, as $n \to \infty$.

One may  study analogous questions
concerning  the structure of the incidence algebras $\sI(\ADset_a[d, n])$ of intervals  
of $a$-floor quotient partial orders.
Each $a$-floor quotient incidence algebra $\sI(\ADaset[1,n])$ can be realized as a noncommutative ring of square integer matrices of size $s_a(n):=|\ADset_a[1,n]|$.
This paper showed for $a \ge 2$ that $s_a(n)$ is smaller than $s(n) = |\ADset_1[1,n]| \approx 2 \sqrt{n}$,
and that the values $s_a(n)$ do not vary smoothly as function of $n$. 
It is an interesting question to estimate the dimension of the incidence algebra $\dim_{\QQ}(\sI(\ADaset[1,n]))$ as a $\QQ$-vector space,
which counts  the total number of incidences, as a function of $a$ and $n$.

%
%
\subsection{Approximate divisor orders interpolating towards the additive total order}
\label{subsec:84}

The paper \cite[Sect. 7.4]{LagR:23a} noted that the $1$-floor quotient order is in some sense ``halfway between'' the multiplicative divisor order and the additive total order on $\NNplus$. 
For a quantitative comparison  of incidence relations, see \cite[Sect. 2.6]{LagR:23a}.
 
The $a$-floor quotient orders provide a sequence of approximate divisor orders interpolating between the $1$-floor quotient order and the divisor order.  
We ask: Is there an interesting family of approximate divisor orders on $\NNplus$ that interpolate in the opposite direction between the $1$-floor quotient order and the additive total order?
 
A positive answer to this question, if there is one, probably will require dropping the condition of monotonicity of convergence of the zeta functions of such a family to the zeta function of the additive order.

\section*{Acknowledgements}
We thank the four reviewers for helpful comments. 
The first author was partially supported by  NSF grant DMS-1701576, 
by a  Chern Professorship at MSRI in Fall 2018,
and by a 2018 Simons Fellowship in Mathematics in January 2019.
MSRI is supported in part by a grant from the National Science Foundation.
The second author was partially supported by
NSF grant DMS-1600223,
a Rackham Predoctoral Fellowship,
and an AMS--Simons Travel Grant.


\bibliographystyle{amsplain}
\bibliography{a-almost-divisors-ref}

\begin{bibdiv}
\begin{biblist}

\bib{AssiG:16}{book}{
      author={Assi, Abdallah},
      author={Garc{\'{\i}}a-S{\'a}nchez, Pedro~A.},
       title={Numerical semigroups and applications},
   publisher={Springer},
     address={Cham},
        date={2016},
        ISBN={978-3-319-41329-7; 978-3-319-41330-3},
}

\bib{Cardinal:10}{article}{
      author={Cardinal, Jean-Paul},
       title={Symmetric matrices related to the {Mertens} function},
        date={2010},
        ISSN={0024-3795},
     journal={Linear Algebra Appl.},
      volume={432},
      number={1},
       pages={161\ndash 172},
}

\bib{CardinalO:20}{article}{
      author={Cardinal, Jean-Paul},
      author={Overholt, Marius},
       title={A variant of {M{\"o}bius} inversion},
        date={2020},
        ISSN={1058-6458},
     journal={Exp. Math.},
      volume={29},
      number={3},
       pages={247\ndash 252},
}

\bib{Heyman:19}{article}{
      author={Heyman, Randell},
       title={Cardinality of a floor function set},
        date={2019},
        ISSN={1553-1732},
     journal={Integers},
      volume={19},
       pages={A67, 7 pp.},
         url={math.colgate.edu/~integers/t67/t67.pdf},
}

\bib{LMR:16}{article}{
      author={Lagarias, Jeffrey~C.},
      author={Murayama, Takumi},
      author={Richman, D.~Harry},
       title={Dilated floor functions that commute},
        date={2016},
        ISSN={0002-9890},
     journal={Amer. Math. Monthly},
      volume={123},
      number={10},
       pages={1033\ndash 1038},
}

\bib{LagR:23a}{article}{
      author={Lagarias, Jeffrey~C.},
      author={Richman, David~Harry},
       title={The floor quotient partial order},
        date={2024},
        ISSN={0196-8858},
     journal={Adv. in Appl. Math.},
      volume={153},
       pages={102615, 54 pp.},
}

\bib{Ore:62}{book}{
      author={Ore, {\O}ystein},
       title={Theory of graphs},
   publisher={American Mathematical Society (AMS)},
     address={Providence, RI},
        date={1962},
}

\bib{RamirezA:05}{book}{
      author={Ram{\'{\i}}rez~Alfons{\'{\i}}n, Jorge~L.},
       title={The {Diophantine} {Frobenius} problem},
   publisher={Oxford University Press},
     address={Oxford},
        date={2005},
        ISBN={0-19-856820-7},
}

\bib{Rota:64}{article}{
      author={Rota, Gian-Carlo},
       title={On the foundations of combinatorial theory. {I}: {Theory} of
  {M{\"o}bius} functions},
        date={1964},
        ISSN={0044-3719},
     journal={Z. Wahrscheinlichkeitstheor. Verw. Geb.},
      volume={2},
       pages={340\ndash 368},
}

\bib{Smith:67}{article}{
      author={Smith, David~A.},
       title={Incidence functions as generalized arithmetic functions. {I}},
        date={1967},
        ISSN={0012-7094},
     journal={Duke Math. J.},
      volume={34},
       pages={617\ndash 633},
}

\bib{Smith:69}{article}{
      author={Smith, David~A.},
       title={Incidence functions as generalized arithmetic functions. {II},
  {III}},
        date={1969},
        ISSN={0012-7094},
     journal={Duke Math. J.},
      volume={36},
       pages={15\ndash 30, 353\ndash 367},
}

\bib{SpiegelO:97}{book}{
      author={Spiegel, Eugene},
      author={O'Donnell, Christopher~J.},
       title={Incidence algebras},
      series={Pure Appl. Math.},
   publisher={Marcel Dekker},
     address={New York, NY},
        date={1997},
      volume={206},
        ISBN={0-8247-0036-8},
}

\bib{Stanley:70}{article}{
      author={Stanley, R.~P.},
       title={Structure of incidence algebras and their automorphism groups},
        date={1970},
        ISSN={0002-9904},
     journal={Bull. Amer. Math. Soc.},
      volume={76},
       pages={1236\ndash 1239},
}

\bib{Stanley:12}{book}{
      author={Stanley, Richard~P.},
       title={Enumerative combinatorics. {Vol}. 1},
     edition={2nd ed.},
   publisher={Cambridge University Press},
     address={Cambridge},
        date={2012},
        ISBN={978-1-107-60262-5; 978-1-107-01542-5; 978-1-139-20056-1},
  url={www.cambridge.org/de/knowledge/isbn/item6832283/?site_locale=de_DE},
}

\bib{Weisner:35}{article}{
      author={Weisner, Louis},
       title={Abstract theory of inversion of finite series},
        date={1935},
        ISSN={0002-9947},
     journal={Trans. Amer. Math. Soc.},
      volume={38},
       pages={474\ndash 484},
}

\end{biblist}
\end{bibdiv}

\end{document}